\documentclass[11pt]{amsart}

\usepackage[english]{babel}
\usepackage[headings]{fullpage}
\usepackage{amsmath,amsfonts,amsthm,amssymb,amscd,mathrsfs,enumerate}
\usepackage[colorlinks]{hyperref}
\usepackage[shortlabels]{enumitem}
\usepackage{bm}

\DeclareMathOperator*{\esssup}{ess\,sup}
\DeclareMathOperator*{\essinf}{ess\,inf}

\DeclareMathOperator{\loc}{loc}
\DeclareMathOperator{\BH}{BHT}
\DeclareMathOperator{\pv}{p.v.}
\DeclareMathOperator{\UMD}{UMD}
\DeclareMathOperator{\BMO}{BMO}

\newtheorem{theorem}{Theorem}[section]
\newtheorem{lemma}[theorem]{Lemma}
\newtheorem{proposition}[theorem]{Proposition}
\newtheorem*{proposition*}{Proposition}
\newtheorem{corollary}[theorem]{Corollary}
\theoremstyle{definition}
\newtheorem{definition}[theorem]{Definition}
\newtheorem{remark}[theorem]{Remark}

\numberwithin{equation}{section}

\newcommand{\N}{\mathbf{N}}
\newcommand{\R}{\mathbf{R}}

\newcommand{\Z}{\mathbf{Z}}
\newcommand{\D}{\mathscr{D}}

\newcommand{\F}{\mathscr{F}}
\newcommand{\Sp}{\mathscr{S}}
\newcommand{\eps}{\varepsilon}

\title[Multilinear weight classes]{Quantitative estimates and extrapolation for multilinear weight classes}
\author{Zoe Nieraeth}
\address{Delft Institute of Applied Mathematics \\ Delft University of Technology \\ P.O. Box 5031\\ 2600 GA Delft \\The Netherlands}
\email{znieraeth@bcamath.org}
\begin{document}
\pagestyle{headings}
\begin{abstract}
In this paper we prove a quantitative multilinear limited range extrapolation theorem which allows us to extrapolate from weighted estimates that include the cases where some of the exponents are infinite. This extends the recent extrapolation result of Li, Martell, and Ombrosi. We also obtain vector-valued estimates including $\ell^\infty$ spaces and, in particular, we are able to reprove all the vector-valued bounds for the bilinear Hilbert transform obtained through the helicoidal method of Benea and Muscalu. Moreover, our result is quantitative and, in particular, allows us to extend quantitative estimates obtained from sparse domination in the Banach space setting to the quasi-Banach space setting.

Our proof does not rely on any off-diagonal extrapolation results and we develop a multilinear version of the Rubio de Francia algorithm adapted to the multisublinear Hardy-Littlewood maximal operator.

As a corollary, we obtain multilinear extrapolation results for some upper and lower endpoints estimates in weak-type and $\BMO$ spaces.
\end{abstract}



%42B20  Singular and oscillatory integrals (Calderón-Zygmund, etc.)
%42B25  Maximal functions, Littlewood-Paley theory

\maketitle

\section{Introduction}
An essential tool in the theory of singular operators is extrapolation. In one of its forms, the classical extrapolation theorem of Rubio de Francia \cite{GR85} says that if an operator $T$ satisfies $L^q(w)$ boundedness for a fixed $q\in(1,\infty)$ and for all weights $w$ in the Muckenhoupt class $A_q$, then $T$ is in fact bounded on $L^p(w)$ for all $p\in(1,\infty)$ and all $w\in A_p$.

Many variations of Rubio de Francia's extrapolation theorem have appeared over the years adapted to various situations. A multilinear version of the extrapolation result was found by Grafakos and Martell \cite{GM04}. Another version provided by Auscher and Martell \cite{AM07} dealt with operators bounded only for a limited range of $p$ rather than for all $p\in(1,\infty)$. Combining these approaches, it was shown by Cruz-Uribe and Martell \cite{CM17} that if there are $0\leq r_j<s_j\leq\infty$ and $q_j\in[r_j,s_j]$, $q_j\neq 0,\infty$, such that an $m$-linear operator $T$ satisfies
\begin{equation}\label{eq:introextrap}
\|T(f_1,\ldots,f_m)\|_{L^q(w^q)}\leq c\prod_{j=1}^m\|f_j\|_{L^{q_j}(w_j^{q_j})}
\end{equation}
for all weights $w_j^{q_j}$ in the restricted Muckenhoupt and Reverse H\"older class $A_{q_j/r_j}\cap RH_{(s_j/q_j)'}$, where $w=\prod_{j=1}^m w_j$, $\frac{1}{q}=\sum_{j=1}^m\frac{1}{q_j}$, then $T$ satisfies the same boundedness for all $p_j\in (r_j,s_j)$ and all $w_j^{p_j}\in A_{p_j/r_j}\cap RH_{(s_j/p_j)'}$, as well as certain vector-valued bounds.

In the linear setting for operators satisfying weighted bounds, it need not be the case that they are bounded on $L^\infty$, as is the case, for example, for the Hilbert transform. In particular, it is impossible to extrapolate estimates to this endpoint. This is in contrast to what happens in the multilinear setting, where it may very well occur that singular integral operators satisfy boundedness as in \eqref{eq:introextrap}, but with some of the $q_j$ being equal to $\infty$. This brings an interest to the question whether it is possible to extrapolate to bounds that include these endpoint cases $p_j=\infty$, starting from an initial weighted estimate where the $q_j$ are also allowed to be infinite. In this work we develop a method that does include these cases based on a multilinear Rubio de Francia algorithm. To facilitate this we give a natural extension in the definition of the weight classes to include these cases, see Definition~\ref{def:weightclass} below. We point out that it is also possible to obtain these endpoint cases through off-diagonal extrapolation methods \cite{LMMOV19}.

As an application for the theory, one can consider the bilinear Hilbert transform $\BH$ given by
\[
\BH(f_1,f_2)(x):=\pv\int_\R\!f_1(x-t)f_2(x+t)\,\frac{\mathrm{d}t}{t},
\]
which plays a central role in the theory of time-frequency analysis. It was shown by Lacey and Thiele \cite{LT99} that $\BH$ is bounded $L^{p_1}\times L^{p_2}\to L^p$ with $\frac{1}{p}=\frac{1}{p_1}+\frac{1}{p_2}$ if $1<p_1,p_2\leq\infty$ and $\frac{2}{3}<p<\infty$. Through the helicoidal method of Benea and Muscalu \cite{BM16,BM18}, vector-valued bounds of the form $L^{p_1}(\ell^{q_1})\times L^{p_2}(\ell^{q_2})\to L^p(\ell^q)$ were established in this range of $p_1$, $p_2$, $p$ for various choices of $1<q_1,q_2\leq\infty$, $\frac{2}{3}<q<\infty$ with $\frac{1}{q}=\frac{1}{q_1}+\frac{1}{q_2}$. However, they left open the problem whether one can obtain vector-valued bounds for all $q_1$, $q_2$, $q$ in the same range as Lacey and Thiele's theorem, i.e., for all $1<q_1,q_2\leq\infty$ with $\frac{2}{3}<q<\infty$. While $\BH$ satisfies weighted bounds as well as more general sparse bounds, see \cite{BM17,CDO16}, the extrapolation result by Cruz-Uribe and Martell \cite{CM17} does not allow one to cover the full range of exponents. In particular, their result cannot retrieve any of the vector-valued bounds involving $\ell^\infty$ spaces. Such bounds also fall outside of the extrapolation result of Lorist and the author \cite{LN17} where vector valued extensions of multilinear operators were considered in the setting of $\UMD$ Banach spaces, since $\ell^\infty$ does not satisfy the $\UMD$ property. The problem seems to be that the multilinear nature of the problem is not completely utilized when one imposes individual conditions on the weights rather than involving an interaction between the various weights.

In the recent work \cite{LMO18} by Li, Martell, and Ombrosi an extrapolation result was presented where they work with a limited range version of the multilinear weight condition introduced by Lerner, Ombrosi, P\'erez, Torres, and Trujillo-Gonz\'alez \cite{LOPTT09} which also appears in \cite{BM17} and, in the bilinear case, in \cite{CDO16}. Indeed, such weight classes are characterized by boundedness of the multi-sublinear Hardy-Littlewood maximal operator as well as by boundedness of sparse forms, meaning the theory can be applied to important operators such as multilinear Calder\'on-Zygmund operators as well as the bilinear Hilbert transform. They introduced the weight class $A_{\vec{p},\vec{r}}$ where $\vec{p}=(p_1,\ldots,p_m)$, $\vec{r}=(r_1,\ldots,r_{m+1})$ and $1\leq r_j\leq p_j<\infty$ and $r'_{m+1}>p$ with $\frac{1}{p}=\sum_{j=1}^m\frac{1}{p_j}$ and $\vec{w}=(w_1,\ldots,w_m)\in A_{\vec{p},\vec{r}}$ if
\begin{equation}\label{eq:wlmo}
[\vec{w}]_{A_{\vec{p},\vec{r}}}:=\sup_{Q\text{ a cube}}\left(\frac{1}{|Q|}\int_Q\!\Big(\prod_{j=1}^m w_j^{\frac{p}{p_j}}\Big)^{\frac{r'_{m+1}}{r'_{m+1}-p}}\,\mathrm{d}x\right)^{\frac{1}{p}-\frac{1}{r'_{m+1}}}\prod_{j=1}^m\left(\frac{1}{|Q|}\int_Q\!w_j^{\frac{r_j}{r_j-p_j}}\,\mathrm{d}x\right)^{\frac{1}{r_j}-\frac{1}{p_j}}<\infty.
\end{equation}
They showed that if \eqref{eq:introextrap} holds for a $\vec{q}$ with $1\leq r_j\leq q_j<\infty$, $r'_{m+1}>q$ and all $(w_1^{q_1},\ldots,w_m^{q_m})\in A_{\vec{q},\vec{r}}$, then $T$ satisfies the same boundedness for all $\vec{p}$ and $(w_1^{p_1},\ldots,w_m^{p_m})\in A_{\vec{p},\vec{r}}$ with $r_j<p_j<\infty$ and $r'_{m+1}>p$. Furthermore, their result extends and reproves some of the vector-valued bounds found by Benea and Muscalu \cite{BM18} for $\BH$. This class of weights does seem to be adapted to the situation even when $p_j=
\infty$, but one needs to be careful in how the constant is interpreted in this case. Similar to the proof of the extrapolation result of Cruz-Uribe and Martell, their proof of this extrapolation result is based upon an off-diagonal extrapolation result, but in their work they left open exactly what happens in the case that some of the exponents are infinite. They announced a paper in which these cases were treated which had not appeared yet when our paper was first posted, but is available now \cite{LMMOV19}. Here they show that, as a feature of off-diagonal extrapolation, it is also possible to obtain estimates that include the cases of infinite exponents.

In this work we again prove an extrapolation result using the multilinear weight classes, and our result includes these endpoint cases which, in particular, include the possibility of extrapolating from the cases where in the initial assumption the exponents can be infinite. Our proof is new and does not rely on any off-diagonal extrapolation result. Rather, we generalize the Rubio de Francia algorithm to a multilinear setting adapted to the multi-sublinear Hardy-Littlewood maximal operator. As a corollary, we are able to obtain vector-valued extensions of operators to spaces including $\ell^\infty$ spaces. Thus, applying this to $\BH$ allows us to recover these endpoint bounds that were obtained earlier through the helicoidal method \cite{BM18}.

Our construction is quantitative in the sense that it allows us to track the dependence of the bounds on the weight constants. Such quantitative versions of extrapolation results were first formalized by Dragi\v{c}evi\'c, Grafakos, Pereyra, and Petermichl in the linear setting in \cite{DGPP05}, but are completely new in the multilinear setting. In the linear setting this result is based on Buckley's sharp weighted bound for the Hardy-Littlewood maximal operator. This bound has been generalized to the multi-sublinear Hardy-Littlewood maximal operator by Dami\'an, Lerner, and P\'erez \cite{DLP15} to a sharp estimate in the setting of a mixed type $A_{\vec{p}}$\,--$A_{\infty}$ estimates and a sharp $A_{\vec{p}}$ bound is found in \cite{LMS14}. We give a different proof of this result for the limited range version of this maximal operator by generalizing a proof of Lerner \cite{Le08}.

Finally, we also show how our quantitative extrapolation result recovers and extends a bound obtained for multi-(sub)linear sparsely dominated operators, generalizing the bound of Hyt\"onen's $A_2$ Theorem \cite{Hy12}. More precisely, sparse domination yields sharp bounds for an operator for exponents $p_1,\ldots,p_m$ only if $\frac{1}{p}=\sum_{j=1}^m\frac{1}{p_m}\leq 1$ so that we may appeal to duality. Our extrapolation result allows us to show that this same control in terms of the weight also holds when $\frac{1}{p}>1$.
\subsection{Symmetry in Muckenhoupt weight classes}
To facilitate our results, we heavily rely on the symmetric structure of the Muckenhoupt classes.

For $p\in(1,\infty)$, a standard method of obtaining weighted $L^p$ estimates with a weight $w$ is by using the duality $(L^p(w))^\ast =L^{p'}(w^{1-p'})$ given through the integral pairing
\[
\langle f,g\rangle=\int_{\R^n}\! fg\,\mathrm{d}x.
\]
Moreover, the Muckenhoupt $A_p$ class is defined through these two weights $w$ and $w^{1-p'}$ through
\[
[w]_{A_p}:=\sup_Q\left(\frac{1}{|Q|}\int_Q\!w\,\mathrm{d}x\right)\left(\frac{1}{|Q|}\int_Q\!w^{1-p'}\,\mathrm{d}x\right)^{p-1}
\]
where the supremum is taken over all cubes $Q\subseteq\R^n$. One way to understand this definition better is by noting that we can relate the weights $w$ and $w^{1-p'}$ through $w^{\frac{1}{p}}(w^{1-p'})^{\frac{1}{p'}}=1$. One can also make sense of this condition if $p=1$ through
\[
[w]_{A_1}:=\sup_Q\left(\frac{1}{|Q|}\int_Q\!w\,\mathrm{d}x\right)\left(\essinf_{x\in Q}w(x)\right)^{-1},
\]
and one usually defines $A_\infty:=\bigcup_{p\in[1,\infty)}A_p$.

When we replace the weight $w$ by the weight $w^p$ we find, using the averaging notation $\langle h\rangle_{q,Q}:=\left(\frac{1}{|Q|}\int_Q\!|h|^q\,\mathrm{d}x\right)^{\frac{1}{q}}$, that
\[
[w^p]_{A_p}^{\frac{1}{p}}=\sup_Q\langle w\rangle_{p,Q}\langle w^{-1}\rangle_{p',Q}
\]
for $p\in (1,\infty)$. The symmetry in this condition is much more prevalent and this condition seems to be more naturally adapted to the weighted $L^p$ theory. Indeed, defining
\[
[w]_p:=[w^p]_{A_p}^{\frac{1}{p}},
\]
we note that $[w]_p=[w^{-1}]_{p'}$. If we denote the Hardy-Littlewood maximal operator by $M$ and if we define the bi-sublinear Hardy-Littlewood maximal operator $M_{(1,1)}$ by
\[
M_{(1,1)}(f_1,f_2)(x):=\sup_{Q\ni x}\langle f_1\rangle_{1,Q}\langle f_2\rangle_{1,Q},
\]
then we have the remarkable equivalences
\begin{equation}\label{eq:prelmaxeq}
\|M_{(1,1)}\|_{L^{p}(w^p)\times L^{p'}(w^{-p'})\to L^{1,\infty}}\eqsim
\|M\|_{L^p(w^p)\to L^{p,\infty}(w^p)}\eqsim\|M\|_{L^{p'}(w^{-p'})\to L^{p',\infty}(w^{-p'})}\eqsim[w]_p,
\end{equation}
where the implicit constant depends only on the dimension, see Proposition~\ref{prop:mainsym} and Proposition~\ref{prop:sinf} below.

Another way of thinking of these equivalences is by setting $w_1:=w$, $w_2:=w^{-1}$ and $p_1:=p$, $p_2:=p'$ so that we have the relations
\begin{equation}\label{eq:2wrelations}
w_1w_2=1,\qquad \frac{1}{p_1}+\frac{1}{p_2}=1.
\end{equation}
Then one can impose a symmetric weight condition
\[
[(w_1,w_2)]_{(p_1,p_2)}:=\sup_{Q}\langle w_1\rangle_{p_1,Q}\langle w_2\rangle_{p_2,Q}<\infty
\]
and note that
\[
[(w_1,w_2)]_{(p_1,p_2)}=[w_1]_{p_1}=[w_2]_{p_2}.
\]
The equivalences in \eqref{eq:prelmaxeq} can now be thought of as
\begin{align*}
\|M_{(1,1)}\|_{L^{p_1}(w_1^{p_1})\times L^{p_2}(w_2^{p_2})\to L^{1,\infty}}&\eqsim[(w_1,w_2)]_{(p_1,p_2)},\\
\|M\|_{L^{p_1}(w_1^{p_1})\to L^{p_1,\infty}(w_1^{p_1})}&\eqsim[w_1]_{p_1},\\
\|M\|_{L^{p_2}(w_2^{p_2})\to L^{p_2,\infty}(w_2^{p_2})}&\eqsim[w_2]_{p_2}.
\end{align*}

We can even make sense of these expressions when $p_1=1$ and $p_2=\infty$ or $p_1=\infty$ and $p_2=1$, given that we use the correct interpretation and this is what allows us to extrapolate using such classes. Indeed, one can think of $f\in L^p(w^p)$ as the condition $\|fw\|_{L^p}<\infty$, which makes sense even when $p=\infty$ by requiring that the function $fw$ is essentially bounded. Using the interpretation $\langle h\rangle_{\infty,Q}=\esssup_{x\in Q}|h(x)|$, we see that the condition $[w_1]_1<\infty$ is equivalent to the usual $A_1$ condition imposed on the weight $w_1=w$, while the condition $[w_1]_\infty<\infty$ is equivalent to the condition $w_2=w^{-1}\in A_1$. We emphasize here that our condition $[w]_\infty<\infty$ is not equivalent to the condition $w\in A_\infty=\bigcup_{p\in[1,\infty)}A_p$ and these notions should not be confused. The condition $w^{-1}\in A_1$ seems to be a natural upper endpoint condition and one can show that this is equivalent to the boundedness
\[
\|(Mf)w\|_{L^\infty}\leq c\|fw\|_{L^\infty},
\]
see Proposition~\ref{prop:sinf} below. It also turns out that this condition allows us to extrapolate away from weighted $L^\infty$ estimates. We point out that this idea has already been used in the endpoint extrapolation result of Harboure, Mac\'ias and Segovia \cite[Theorem 3]{HMS88}.

We wish to view our symmetric weight condition in the context of extrapolation. In proving Rubio de Francia's extrapolation theorem, one usually starts with a pair of functions $(h,f)$ and assumes that one has the inequality
\begin{equation}\label{eq:heurexintro}
\|h\|_{L^q(w^q)}\leq c\|f\|_{L^q(w^q)}
\end{equation}
for some $q\in[1,\infty]$ and all weights $w$ satisfying $[w]_q<\infty$. The idea is then that given a $p\in (1,\infty)$ and a weight $w$ satisfying $[w]_p<\infty$, one can construct a weight $W$, possibly depending on $f$, $h$, and $w$, so that $W$ satisfies $[W]_q<\infty$ as well as some additional properties to ensure that we can use \eqref{eq:heurexintro} with $W$ to conclude that
\begin{equation}\label{eq:heurexintro2}
\|h\|_{L^p(w^p)}\leq \tilde{c}\|f\|_{L^p(w^p)}.
\end{equation}
Applying this with $h=Tf$ then gives the desired boundedness for an operator $T$. For the proof one usually splits into two cases, namely the case where $p<q$ and the case where $p>q$. In the former case one can apply H\"older's inequality to move from $L^p$ to $L^q$ and in the latter case one uses duality and a similar trick to move from $L^{p'}$ to $L^{q'}$. The point is that both of these cases are essentially the same, but due to the notation we use we have to deal with the cases separately. Here, we wish to come up with a formalization to avoid this redundancy.

The extrapolation theorem is essentially a consequence of to the following proposition:
\begin{proposition*}
Suppose we are given $p_1, p_2\in(1,\infty)$ satisfying $\frac{1}{p_1}+\frac{1}{p_2}=1$ and weights $w_1,w_2$ satisfying $w_1w_2=1$ and $[(w_1,w_2)]_{(p_1,p_2)}<\infty$. Moreover, assume we have two functions $f_1\in L^{p_1}(w_1^{p_1})$ and $f_2\in L^{p_2}(w_2^{p_2})$ and $q_1, q_2\in[1,\infty]$ with $\frac{1}{q_1}+\frac{1}{q_2}=1$. Then there are weights $W_1$, $W_2$ satisfying $W_1W_2=1$,
\[
\|f_1 W_1\|_{L^{q_1}}\|f_2 W_2\|_{L^{q_2}}\leq 2\|f_1 w_1\|_{L^{p_1}}\|f_2 w_2\|_{L^{p_2}}
\]
and
\[
[(W_1,W_2)]_{(q_1,q_2)}\leq C[(w_1,w_2)]_{(p_1,p_2)}^{\max\left(\frac{p_1}{q_1},\frac{p_2}{q_2}\right)}.
\]
\end{proposition*}
Indeed, the result of the extrapolation theorem follows by applying the proposition with $f_1:=f$, $q_1:=q$, $q_2:=q'$, $p_1:=p$, $p_2:=p'$, $w_1:=w$, $w_2=w^{-1}$ and $W_1:=W$, $W_2:=W^{-1}$ so that, by \eqref{eq:heurexintro}, we have
\[
|\langle h,f_2\rangle|\leq\|hW\|_{L^q}\|f_2W^{-1}\|_{L^{q'}}\leq c\|fW\|_{L^q}\|f_2W^{-1}\|_{L^{q'}}\leq 2c\|f w\|_{L^{p}}\|f_2 w^{-1}\|_{L^{p'}}.
\]
Thus, by duality, we obtain \eqref{eq:heurexintro2}, as desired.

The proof of the proposition uses the classical construction using the Rubio de Francia algorithm and the novelty here is our symmetric formulation. A proof can be found in this work, as it is a special case of Theorem \ref{thm:main}. The case $p<q$ in the proposition takes the form $p_1<q_1$ and $p_2>q_2$ while the case $p>q$ takes the form $p_1>q_1$ and $p_2<q_2$. The fact that the proposition is formulated completely symmetrically in terms of the parameters indexed over $\{1,2\}$, where we note that $[(w_1,w_2)]_{(p_1,p_2)}=[(w_2,w_1)]_{(p_2,p_1)}$, means that these respective cases can be proven using precisely the same argument, up to a permutation of the indices. Thus, without loss of generality, one only needs to prove one of the two cases.

These symmetries become especially important in the $m$-linear setting where we are dealing with parameters indexed over $\{1,\ldots,m+1\}$ and the amount of cases we have to consider increases. Thanks to our formulation, we will be able to reduce these multiple cases back to a single case in our arguments again by permuting the indices.

We wish to point out here that to facilitate our symmetric formulation and to use the duality argument involving the Rubio de Francia algorithm as above, we need to essentially restrict ourselves to the Banach range $\frac{1}{p}\leq 1$. However, in the $m$-linear setting one also has to deal with the quasi-Banach range $\frac{1}{p}>1$. This means that to employ our multilinear Rubio de Francia algorithm, we must first reduce to the case where $\sum_{j=1}^m\frac{1}{p_j}=\frac{1}{p}\leq1.$ In this case we can set $\frac{1}{p_{m+1}}:=1-\frac{1}{p}\geq 0$ and $\sum_{j=1}^{m+1}\frac{1}{p_j}=1$, which places us in the setting of Theorem \ref{thm:main}. This is not a problem however, as reducing to this case is facilitated by the rescaling properties of the multilinear weight classes, see also Remark \ref{rem:qbvsb}. In conclusion, even though our multilinear Rubio de Francia algorithm is applied in the Banach range $\frac{1}{p}\leq 1$, our result also includes the quasi-Banach range $\frac{1}{p}>1$.

This article is organized as follows:
\begin{itemize}
\item In Section~\ref{sec:prelim} we state our main result and give an overview of the multilinear weight classes, proving some important properties as well as proving new quantitative estimates with respect to the multisublinear maximal operator as well as sparse forms.

\item In Section~\ref{sec:proof} we prove the main result.

\item In Section~\ref{sec:app} we apply the extrapolation result for weak type bounds and certain $\BMO$ type bounds as well as for vector-valued bounds. Moreover, we give an application of our results to the bilinear Hilbert transform.
\end{itemize}
\textbf{Acknowledgements.} The author would like to thank Cristina Benea and Dorothee Frey for their feedback on the draft. In particular the author is grateful to Cristina for her thoroughness in helping to provide context regarding what is known for the bilinear Hilbert transform and multilinear extrapolation. Finally, the author would like to thank the anonymous referee for their helpful suggestions.
\section{Multilinear weight classes}\label{sec:prelim}
\subsection{Setting and main result}
We work in $\R^n$ equipped with the Lebesgue measure $\mathrm{d}x$. This is mostly for notational convenience and our results also hold in the more general setting of spaces of doubling quasimetric measure spaces, provided one uses the right notion of dyadic cubes in this setting, see \cite{hytonen12}. For a measurable set $E$ we denote its Lebesgue measure by $|E|$. A measurable function $w:\R^n\to (0,\infty)$ is called a weight. We can identify $w$ with a measure by $w(E):=\int_E\!w\,\mathrm{d}x$. For $p\in(0,\infty]$, a weight $w$, and a measurable function $f$ on $\R^n$ we say that $f\in L^p(w^p)$ provided that $\|f\|_{L^p(w^p)}:=\|fw\|_{L^p}<\infty$. Moreover, for a measurable set $E\subseteq\R^n$ with $0<|E|<\infty$ we write
\[
\langle f\rangle_{p,E}:=\left(\frac{1}{|E|}\int_E\!|f|^p\,\mathrm{d}x\right)^{\frac{1}{p}}
\]
when $0<p<\infty$ and $\langle f\rangle_{\infty,E}:=\esssup_{x\in E}|f(x)|$.

We will use the notation $A\lesssim B$ if there is a constant $c>0$, independent of the important parameters, such that $A\leq cB$. Moreover, we write $A\eqsim B$ if $A\lesssim B$ and $B\lesssim A$.

Let $m\in\N$ and let $r_1,\ldots,r_m\in(0,\infty)$, $s\in(0,\infty]$. For $p_1,\ldots,p_m\in(0,\infty]$, writing $\vec{r}=(r_1,\ldots,r_m)$ and similarly for $\vec{p}$, we write $\vec{r}\leq\vec{p}$ if $r_j\leq p_j\leq\infty$ for all $j\in\{1,\ldots,m\}$. Moreover, we write $(\vec{r},s)\leq\vec{p}$ if $\vec{r}\leq\vec{p}$ and $p\leq s$, where $p$ is defined by
\[
\frac{1}{p}=\sum_{j=1}^m\frac{1}{p_j}.
\]
Similarly, we write $\vec{r}<\vec{p}$ if $r_j<p_j$ for all $j\in\{1,\ldots,m\}$ and we write $(\vec{r},s)<\vec{p}$ if $\vec{r}<\vec{p}$ and $p<s$.

\begin{definition}\label{def:weightclass}
Let $r_1,\ldots,r_m\in(0,\infty)$, $s\in(0,\infty]$, and $p_1,\ldots,p_m\in(0,\infty]$ with $(\vec{r},s)\leq\vec{p}$. Let $w_1,\ldots, w_m$ be weights and write $w=\prod_{j=1}^m w_j$, $\vec{w}=(w_1,\ldots,w_m)$. We say that $\vec{w}\in A_{\vec{p},(\vec{r},s)}$ if
\[
[\vec{w}]_{\vec{p},(\vec{r},s)}:=\sup_Q\left(\prod_{j=1}^m\langle w_j^{-1}\rangle_{\frac{1}{\frac{1}{r_j}-\frac{1}{p_j}},Q}\right)\langle w\rangle_{\frac{1}{\frac{1}{p}-\frac{1}{s}},Q}<\infty,
\]
where the supremum is taken over all cubes $Q\subseteq\R^n$.
\end{definition}
As a point of comparison, we note here that, for finite $p_j$, our condition $\vec{w}\in A_{\vec{p},(\vec{r},s)}$ is equivalent to the condition $(w_1^{p_1},\ldots,w_m^{p_m})\in A_{\vec{p},(r_1,\ldots,r_m,s')}$, where the latter condition considers the weight class of Li, Martell, and Ombrosi defined in \eqref{eq:wlmo}. Thus, in this range their extrapolation result \cite{LMO18} consider the same weights as we do.

Our main theorem is as follows:
\begin{theorem}[Quantitative multilinear limited range extrapolation]\label{thm:qmlre}
Let $(f_1,\ldots,f_m,h)$ be an $m+1$-tuple of measurable functions and let $r_1,\ldots,r_m\in(0,\infty)$, $s\in(0,\infty]$. Suppose that for some $q_1,\ldots,q_m\in(0,\infty]$ with $\vec{q}\geq(\vec{r},s)$ there is an increasing function $\phi_{\vec{q}}$ such that
\begin{equation}\label{eq:multextrapinit}
\|h\|_{L^q(w^q)}\leq\phi_{\vec{q}}([\vec{w}]_{\vec{q},(\vec{r},s)})\prod_{j=1}^m\|f_j\|_{L^{q_j}(w_j^{q_j})}
\end{equation}
for all $\vec{w}\in A_{\vec{q},(\vec{r},s)}$.

Then for all $p_1\,\ldots,p_m\in(0,\infty]$ with $\vec{p}>(\vec{r},s)$ there is an increasing function $\phi_{\vec{p},\vec{q},\vec{r},s}$ such that
\begin{equation}\label{eq:multextrapend}
\|h\|_{L^p(w^p)}\leq\phi_{\vec{p},\vec{q},\vec{r},s}([\vec{w}]_{\vec{p},(\vec{r},s)})\prod_{j=1}^m\|f_j\|_{L^{p_j}(w_j^{p_j})}
\end{equation}
for all $\vec{w}\in A_{\vec{p},(\vec{r},s)}$. More explicitly, we can take
\begin{equation}\label{eq:multextraquant}
\phi_{\vec{p},\vec{q},\vec{r},s}(t)=2^{\frac{m^2}{r}}\phi_{\vec{q}}\Big(C_{\vec{p},\vec{q},\vec{r},s}t^{ r\max\left(\frac{\frac{1}{r_1}-\frac{1}{q_1}}{\frac{1}{r_1}-\frac{1}{p_1}},\ldots,\frac{\frac{1}{r_m}-\frac{1}{q_m}}{\frac{1}{r_m}-\frac{1}{p_m}},\frac{\frac{1}{q}-\frac{1}{s}}{\frac{1}{p}-\frac{1}{s}}\right)}\Big)^{\frac{1}{r}},
\end{equation}
where $\frac{1}{r}=\sum_{j=1}^m\frac{1}{r_j}$.
\end{theorem}
We note that if there is equality in one of the components in $\vec{q}\geq(\vec{r},s)$, i.e., if $q=s$ or $q_j=r_j$ for some $j\in\{1,\ldots,m\}$, then we may also include the respective cases with $p=s$ or $p_j=r_j$ to the conclusion of the extrapolation result. In this case one should respectively use the interpretation $\frac{\frac{1}{q}-\frac{1}{s}}{\frac{1}{p}-\frac{1}{s}}=1$ or $\frac{\frac{1}{r_j}-\frac{1}{q_j}}{\frac{1}{r_j}-\frac{1}{p_j}}=1$. To see this, one need only note that the proof we give of the theorem already accounts for the respective cases when $\frac{1}{p}=\frac{1}{q}$ or $\frac{1}{p_j}=\frac{1}{q_j}$.

Our result is stronger than that in \cite{LMO18} in the sense that we do not have to restrict our exponents to the case where they are finite, i.e., in the initial assumption we include all the cases where $q_j=\infty$ and in the conclusion we similarly obtain all the cases where $p_j=\infty$, see also \cite{LMMOV19}. We emphasize here that we use the interpretation $\|f_j\|_{L^{q_j}(w_j^{q_j})}=\|f_jw_j\|_{L^\infty}$ in the case where $q_j=\infty$ and we need to impose the weight condition from Definition~\ref{def:weightclass} with $\frac{1}{q_j}=0$. For example, in the case $m=1$, $r=1$ and $q=s=\infty$, one has to use the condition $w\in A_{\infty,(1,\infty)}$ in the initial estimate \eqref{def:weightclass} which, following our definition, is equivalent to the condition $w^{-1}\in A_1$. This stronger result is possible due to our use of a multilinear Rubio de Francia algorithm, fully utilizing the multilinear nature of the problem. Our result also implies vector valued estimates in these ranges and we refer the reader to Section \ref{sec:app} where we elaborate on this further.

Next we make some remarks on the quantitative result \eqref{eq:multextraquant}.

Usually in applications, the increasing function will be of the form $\phi_{\vec{q}}(t)=c t^{\alpha}$ for some $c,\alpha>0$. Then we find from \eqref{eq:multextraquant} that
\[
\phi_{\vec{p},\vec{q},\vec{r},s}(t)=\tilde{c}t^{\alpha\max\left(\frac{\frac{1}{r_1}-\frac{1}{q_1}}{\frac{1}{r_1}-\frac{1}{p_1}},\ldots,\frac{\frac{1}{r_m}-\frac{1}{q_m}}{\frac{1}{r_m}-\frac{1}{p_m}},\frac{\frac{1}{q}-\frac{1}{s}}{\frac{1}{p}-\frac{1}{s}}\right)}.
\]
In the case $m=1$, $r=1$, $s=\infty$, this means that we have
\begin{equation}\label{eq:sharpaqapp}
\phi_{p,q,1,\infty}(t)=\tilde{c}t^{\alpha\max\left(\frac{p'}{q'},\frac{p}{q}\right)},
\end{equation}
and this coincides with the bound obtained in \cite{DGPP05}. This result was used in Hyt\"onen's $A_2$ theorem \cite{Hy12} to reduce proving the sharp estimate
\begin{equation}\label{eq:A2thm}
\|Tf\|_{L^p(w^p)}\lesssim[w^p]^{\frac{\max(p',p)}{p}}_{A_p}\|f\|_{L^p(w^p)}
\end{equation}
for Calder\'on-Zygmund operators $T$ to only having to prove the linear $A_2$ bound
\[
\|Tf\|_{L^2(w^2)}\lesssim[w^2]_{A_2}\|f\|_{L^2(w^2)}.
\]
Indeed, noting that $[w]_{p,(1,\infty)}=[w^p]_{A_p}^{\frac{1}{p}}$, we find that \eqref{eq:A2thm} follows from \eqref{eq:sharpaqapp} by taking $\alpha=2$ and $q=2$.

The fact that we need to extrapolate from $q=2$ to obtain the sharp bounds for Calder\'on-Zygmund operators speaks to their nature as operators revolving around their properties in $L^2$. As a contrast, we note that the estimate
\[
\|(Tf)w\|_{L^\infty}\lesssim[w^{-1}]_{A_1}\|fw\|_{L^\infty}
\]
is central, for example, for when $T$ is the Hardy-Littlewood maximal operator $M$. Indeed, by \eqref{eq:sharpaqapp} with $q=\infty$ and $\alpha=1$ and by noting that $[w]_{\infty,(1,\infty)}=[w^{-1}]_{A_1}$, this estimate extrapolates to the estimate
\[
\|Tf\|_{L^p}\lesssim[w^p]^{\frac{p'}{p}}_{A_p}\|f\|_{L^p(w^p)}
\]
for $p\in(1,\infty]$, which is precisely Buckley's sharp bound obtained for $M$. We point out here that this argument is actually circular for when $T=M$, since the proof of the quantitative estimate in the extrapolation result makes use of Buckley's sharp bound. Nonetheless, we think this example is heuristically interesting, since it exhibits how one can extrapolate away from weighted $L^\infty$ estimates. Multilinear versions of Buckley's sharp bound have been found in \cite{DLP15, LMS14} and can be recovered in a similar way, see also Theorem~\ref{thm:upendextrap}.

The remainder of this section will be dedicated to a discussion on the quantitative properties of the multilinear weight classes. We split this into two separate cases. In the first case we adopt the symmetric notation from the introduction and think in terms of $m+1$-tuples of weights and parameters satisfying a symmetric relation. In the second case we adopt the more classical approach of thinking in terms of $m$-tuples and we prove some key results for our main theorem.
\subsection{Quantitative properties of multilinear weight classes: the $m+1$-tuple case}\label{sec:twotwo}
Let $r_1,\ldots,r_m\in(0,\infty)$, $s\in(0,\infty]$, and $p_1,\ldots,p_m\in(0,\infty]$ and let $w_1,\ldots, w_m$ be weights, with $w:=\prod_{j=1}^m w_j$. In terms of symmetries, the definition of the weight class
\[
[\vec{w}]_{\vec{p},(\vec{r},s)}=\sup_Q\left(\prod_{j=1}^m\langle w_j^{-1}\rangle_{\frac{1}{\frac{1}{r_j}-\frac{1}{p_j}},Q}\right)\langle w\rangle_{\frac{1}{\frac{1}{p}-\frac{1}{s}},Q}
\]
seems to be best suited to the case where $\frac{1}{p}\leq1$. Indeed, if we set $\frac{1}{p_{m+1}}:=1-\frac{1}{p}\geq 0$, $\frac{1}{r_{m+1}}:=1-\frac{1}{s}$ and $w_{m+1}:=w^{-1}$, then we have
\[
\sum_{j=1}^{m+1}\frac{1}{p_j}=1,\qquad\prod_{j=1}^{m+1}w_j=1.
\]
The condition $(\vec{r},s)\leq\vec{p}$ is equivalent to $r_j\leq p_j$ for all $j\in\{1,\ldots,m+1\}$ and the constant for the weight class now takes the form
\[
[\vec{w}]_{\vec{p},(\vec{r},s)}=\sup_Q\prod_{j=1}^{m+1}\langle w_j^{-1}\rangle_{\frac{1}{\frac{1}{r_j}-\frac{1}{p_j}},Q}=[(w_1,\ldots,w_{m+1})]_{(p_1,\ldots,p_{m+1}),((r_1,\ldots,r_{m+1}),\infty)},
\]
where the last equality follows from the fact that the term involving the product weight in the $m+1$-linear weight class is equal to $1$. The symmetry of this last expression also emphasizes a certain permutational invariance. Indeed, if $\pi\in S_{m+1}$ is a permutation, then, since
\[
\sum_{j=1}^{m+1}\frac{1}{p_{\pi(j)}}=\sum_{j=1}^{m+1}\frac{1}{p_j}=1,\qquad\prod_{j=1}^{m+1}w_{\pi(j)}=\prod_{j=1}^{m+1}w_j=1,
\]
we have
\[
[\vec{w}]_{\vec{p},(\vec{r},s)}=[(w_{\pi(1)},\ldots,w_{\pi(m)})]_{(p_{\pi(1)},\ldots,p_{\pi(m)}),((r_{\pi(1)},\ldots,r_{\pi(m)}),r'_{\pi(m+1)})},
\]
and this will be used in the proof of our extrapolation theorem.
\begin{remark}\label{rem:qbvsb}
While we restrict ourselves to the Banach range $\frac{1}{p}\leq 1$ in this section, we do point out that our main results do also apply in the cases where $\frac{1}{p}>1$. This is facilitated by the rescaling property
\[
[\vec{w}]^{\frac{1}{\alpha}}_{\frac{\vec{p}}{\alpha},(\frac{\vec{r}}{\alpha},\frac{s}{\alpha})}=[(w_1^{\frac{1}{\alpha}},\ldots,w_m^{\frac{1}{\alpha}})]_{\vec{p},(\vec{r},s)},
\]
which, in our arguments, allows us to reduce back to the case where $\frac{1}{p}\leq 1$, see also the proof of Theorem \ref{thm:qmlre}.
\end{remark}

It will sometimes also be useful to redefine $v_j:=w_j^{-\frac{1}{\frac{1}{r_j}-\frac{1}{p_j}}}$ for $j\in\{1,\ldots,m+1\}$ so that
\[
[(w_1,\ldots,w_{m+1})]_{(p_1,\ldots,p_{m+1}),((r_1,\ldots,r_{m+1}),\infty)}=\sup_Q\prod_{j=1}^{m+1}\langle v_j\rangle^{\frac{1}{r_j}-\frac{1}{p_j}}_{1,Q}.
\]

These weight classes are governed by a certain maximal operator, see also \cite{LOPTT09}.
\begin{definition}
Given $r_1,\ldots,r_m\in(0,\infty)$, we define the $m$-sublinear Hardy-Littlewood maximal operator
\[
M_{\vec{r}}(f_1,\ldots,f_m)(x):=\sup_{Q\ni x}\prod_{j=1}^m\langle f_j\rangle_{r_j,Q}
\]
for $f_j\in L^{r_j}_{\loc}$, where the supremum is taken over all cubes $Q\subseteq\R^n$ containing $x$. Moreover, for a dyadic grid $\D$ we define
\[
M^{\D}_{\vec{r}}(f_1,\ldots,f_m)(x):=\sup_{\substack{Q\ni x\\Q\in\D}}\prod_{j=1}^m\langle f_j\rangle_{r_j,Q}
\]
for $f_j\in L^{r_j}_{\loc}$.
\end{definition}
For the relevant definitions and results regarding dyadic grids we refer the reader to \cite{LN15}. A property we need is the fact that there exist $3^n$ dyadic grids $(\D^\alpha)_{\alpha=1}^{3^n}$ such that for each cube $Q\subseteq\R^n$ there is an $\alpha$ and a cube $\tilde{Q}\in\D^\alpha$ such that $Q\subseteq \tilde{Q}$ and $|\tilde{Q}|\leq 6^n|Q|$. This implies the following:
\begin{lemma}\label{lem:dyadicmax}
Let $r_1,\ldots,r_m\in(0,\infty)$. Then there exist $3^n$ dyadic grids $(\D^\alpha)_{\alpha=1}^{3^n}$ such that
\[
M_{\vec{r}}\lesssim\sum_{\alpha=1}^{3^n}M_{\vec{r}}^{\D^\alpha}.
\]
\end{lemma}
See also \cite{LN15}.
\begin{definition}
A collection of cubes $\Sp$ in a dyadic grid is called \emph{sparse} if there is a pairwise disjoint collection of measurable sets $(E_Q)_{Q\in\Sp}$ such that $E_Q\subseteq Q$ and $|Q|\leq 2|E_Q|$.

Given $r_1,\ldots,r_m\in(0,\infty)$, for a sparse collection of cubes $\Sp$ we define the sparse operator
\[
A_{\vec{r},\Sp}(f_1,\ldots,f_m):=\sum_{Q\in\Sp}\left(\prod_{j=1}^m\langle f_j\rangle_{r_j,Q}\right)\chi_Q,
\]
and the sparse form
\[
\Lambda_{\vec{r},\Sp}(f_1,\ldots,f_m):=\sum_{Q\in\Sp}\prod_{j=1}^m\langle f_j\rangle_{r_j,Q}|Q|.
\]
\end{definition}
We point out here that the sparsity constant $2$ appearing in the estimate $|Q|\leq 2|E_Q|$ is not too important and in most situations it can be replaced by any other constant greater than $1$. Note however, that we will be considering the form $\sup_{\Sp}\Lambda_{\vec{r},\Sp}$ and here it is important that one only considers sparse collections in this supremum with the same sparsity constant. See \cite{LN15} for further properties and results regarding sparse collections of cubes.

Since this section contains results involving both $m$-tuples and $m+1$-tuples with the same parameters, it is convenient to separate these notationally. We will use the following convention: for $m+1$ parameters $\alpha_1,\ldots,\alpha_{m+1}$ we shall use the boldface notation $\pmb{\alpha}=(\alpha_1,\ldots,\alpha_{m+1})$ for $m+1$-tuples while we will use the arrow notation $\vec{\alpha}=(\alpha_1,\ldots,\alpha_m)$ for $m$-tuples.

The main result for this section is the following:
\begin{proposition}\label{prop:mainsym}
Let $r_1,\ldots,r_{m+1}\in(0,\infty)$, $p_1,\ldots,p_{m+1}\in(0,\infty]$ satisfy $\frac{1}{p_j}<\frac{1}{r_j}$ for all $j\in\{1,\ldots,m+1\}$ and $\sum_{j=1}^{m+1}\frac{1}{p_j}=1$. Moreover, let $w_1,\ldots, w_{m+1}$ be weights satisfying $\prod_{j=1}^{m+1}w_j=1$. Then the following are equivalent:
\begin{enumerate}[(i)]
\item \label{it:equ1} $\pmb{w}\in A_{\pmb{p},(\pmb{r},\infty)}$;
\item \label{it:equ2} $\|M_{\pmb{r}}\|_{L^{p_1}(w_1^{p_1})\times\cdots\times L^{p_{m+1}}(w_{m+1}^{p_{m+1}})\to L^{1,\infty}}<\infty$;
\item \label{it:equ3} $\|M_{\pmb{r}}\|_{L^{p_1}(w_1^{p_1})\times\cdots\times L^{p_{m+1}}(w_{m+1}^{p_{m+1}})\to L^1}<\infty$;
\item \label{it:equ4} $\|\sup_{\Sp}\Lambda_{\pmb{r},\Sp}\|_{L^{p_1}(w_1^{p_1})\times\cdots\times L^{p_{m+1}}(w_{m+1}^{p_{m+1}})\to\R}<\infty$.
\end{enumerate}
Moreover, we have
\begin{align}\label{eq:aleq1}
&\|M_{\pmb{r}}\|_{L^{p_1}(w_1^{p_1})\times\cdots\times L^{p_{m+1}}(w_{m+1}^{p_{m+1}})\to L^{1,\infty}}\eqsim[\pmb{w}]_{\pmb{p},(\pmb{r},\infty)},\\\label{eq:aleq2} &\|M_{\pmb{r}}\|_{L^{p_1}(w_1^{p_1})\times\cdots\times L^{p_{m+1}}(w_{m+1}^{p_{m+1}})\to L^1}\eqsim\|\sup_{\Sp}\Lambda_{\pmb{r},\Sp}\|_{L^{p_1}(w_1^{p_1})\times\cdots\times L^{p_{m+1}}(w_{m+1}^{p_{m+1}})\to\R}
\end{align}
where the implicit constants depend only on the dimension, and
\begin{equation}\label{eq:aleq3}
\|\sup_{\Sp}\Lambda_{\pmb{r},\Sp}\|_{L^{p_1}(w_1^{p_1})\times\cdots\times L^{p_{m+1}}(w_{m+1}^{p_{m+1}})\to\R}\lesssim c_{\pmb{p},\pmb{r}}[\pmb{w}]^{\max_{j=1,\ldots,m+1}\left\{\frac{\frac{1}{r_j}}{\frac{1}{r_j}-\frac{1}{p_j}}\right\}}_{\pmb{p},(\pmb{r},\infty)},
\end{equation}
where the implicit constant depends on the dimension and
\[
c_{\pmb{p},\pmb{r}}=\prod_{j=1}^{m+1} \left[\frac{\frac{1}{r_j}}{\frac{1}{r_j}-\frac{1}{p_j}}\right]^{\frac{1}{r_j}}.
\]
\end{proposition}
\begin{remark}
We again point out that the condition $\pmb{w}\in A_{\pmb{p},(\pmb{r},\infty)}$ is equivalent to the condition $\vec{w}\in A_{\vec{p},(\vec{r},r_{m+1}')}$, with equal constants. Moreover, the results containing the sparse forms are formulated with the supremum taken inside of the norm. One can equivalently put the supremum outside of the norm which follows from the fact that there is a single sparse form that dominates all the other sparse forms, see \cite[Section 4]{LM17}.
\end{remark}
In the case $m=1$, $r_1=r_2=1$, the equivalence \eqref{eq:aleq1} takes the more familiar form
\[
\|M_{(1,1)}\|_{L^p(w^p)\times L^{p'}(w^{-p'})\to L^{1,\infty}}\eqsim[w^p]_{A_p}^{\frac{1}{p}}
\]
which appeared in the introduction.

We note that the estimate \eqref{eq:aleq3} was already obtained in \cite{CDO16} in the case $m=2$.

For $r_1=r$, $r_2=s'$ the estimate \eqref{eq:aleq3} takes the form
\begin{equation}\label{eq:bfpest}
\|\sup_{\Sp}\Lambda_{(r,s'),\Sp}\|_{L^p(w^p)\times L^{p'}(w^{-p'})\to\R}\lesssim \left[w^{\left(\frac{1}{p}-\frac{1}{s}\right)^{-1}}\right]_{A_{\frac{\frac{1}{r}-\frac{1}{s}}{\frac{1}{p}-\frac{1}{s}}}}^{\max\left(\frac{\frac{1}{p}-\frac{1}{s}}{\frac{1}{r}-\frac{1}{p}}\frac{1}{r},\frac{1}{s'}\right)}
\end{equation}
and when $r=1$ and $s=\infty$ we reobtain the sharp bound from the $A_2$ theorem. We wish to compare \eqref{eq:bfpest} to the bound obtained in \cite{frey16}. For their main result they prove that
\begin{equation}\label{eq:bfpest2}
\|\sup_{\Sp}\Lambda_{(r,s'),\Sp}\|_{L^p(w)\times L^{p'}(w^{1-p'})\to\R}\lesssim\Big([w]_{A_{\frac{p}{r}}}[w]_{RH_{\left(\frac{s}{p}\right)'}}\Big)^{\max\left(\frac{1}{p-r},\frac{s-1}{s-p}\right)},
\end{equation}
Our result implies that
\[
\|\sup_{\Sp}\Lambda_{(r,s'),\Sp}\|_{L^p(w)\times L^{p'}(w^{1-p'})\to\R}\lesssim \Big([w]_{A_{\frac{p}{r}}}[w]_{RH_{\left(\frac{s}{p}\right)'}}\Big)^{\left(\frac{s}{p}\right)'\max\left(\frac{\frac{1}{p}-\frac{1}{s}}{\frac{1}{r}-\frac{1}{p}}\frac{1}{r},\frac{1}{s'}\right)},
\]
see also \cite{JN91}, and this recovers the estimate \eqref{eq:bfpest2}.

Finally, we point out here that the estimate \eqref{eq:aleq3} already appears in \cite[p.~12]{LMO18} for the particular choice $\frac{1}{p_j}=\frac{1}{r_j}\frac{1}{\sum_{j=1}^{m+1}\frac{1}{r_j}}$, and it seems like this choice of $p_j$ is central for the theory of these sparse forms, see also the proof of Corollary~\ref{cor:sparseext}.

For the proof of the proposition we will require several preparatory lemmata.
\begin{lemma}\label{lem:sparsedommax}
Let $0<r_1,\ldots,r_m<\infty$. Then for each dyadic grid $\D$ and all $f_j\in L^{r_j}$ there is a sparse collection $\Sp\subseteq\D$ such that
\[
M^{\D}_{\vec{r}}(f_1,\ldots,f_m)\leq 2^{\frac{n+1}{r}}\sum_{Q\in\Sp}\prod_{j=1}^m\langle f_j\rangle_{r_j,Q}\chi_{E_Q}
\]
pointwise almost everywhere, where $\frac{1}{r}=\sum_{j=1}^m\frac{1}{r_j}$. In particular we have
\[
M^{\D}_{\vec{r}}(f_1,\ldots,f_m)\leq2^{\frac{n+1}{r}}A_{\vec{r},\Sp}(f_1,\ldots,f_m)
\]
pointwise almost everywhere.
\end{lemma}
The proof is essentially the same as the well-known result in the case $m=1$, $r=1$.
\begin{proof}
For $k\in\Z$ we define
\[
\Omega_k:=\{x\in\R^n:M^{\D}_{\vec{r}}(f_1,\ldots,f_m)(x)>2^{\frac{n+1}{r}k}\}.
\]
By taking the maximal cubes $Q$ in $\Omega_k$ we obtain a pairwise disjoint collection $\mathcal{Q}_k\subseteq\D$ such that $\Omega_k=\bigcup_{Q\in\mathcal{Q}_k}Q$ and
\begin{equation}\label{eq:czdspm}
 2^{\frac{n+1}{r}k}<\prod_{j=1}^m\langle f_j\rangle_{r_j,Q}\leq\frac{2^{\frac{n+1}{r}(k+1)}}{2^{\frac{1}{r}}}
\end{equation}
for all $Q\in\mathcal{Q}_k$. We define $\Sp:=\cup_{k\in\Z}\mathcal{Q}_k$ and claim that $\Sp$ is a sparse collection of cubes. Indeed, for $Q\in\mathcal{Q}_k$ it follows from \eqref{eq:czdspm} that for any $Q'\in\mathcal{Q}_{k+1}$ we have
\[
\prod_{j=1}^m\langle f_j\rangle_{r_j,Q'}>2^{\frac{1}{r}}\frac{2^{\frac{n+1}{r}(k+1)}}{2^{\frac{1}{r}}}\geq2^{\frac{1}{r}}\prod_{j=1}^m\langle f_j\rangle_{r_j,Q}.
\]
Thus, by maximality of $\mathcal{Q}_k$ and H\"older's inequality with $\sum_{j=1}^m\frac{r}{r_j}=1$, we have
\begin{align*}
|\Omega_{k+1}\cap Q|&=\sum_{\substack{Q'\in\mathcal{Q}_{k+1}\\ Q'\subseteq Q}}|Q'|\leq\frac{1}{2\prod_{j=1}^m\langle f_j\rangle^r_{r_j,Q}}\sum_{\substack{Q'\in\mathcal{Q}_{k+1}\\ Q'\subseteq Q}}|Q'|\prod_{j=1}^m\langle f_j\rangle^r_{r_j,Q'}\\
&=\frac{|Q|}{2}\frac{\displaystyle\sum_{\substack{Q'\in\mathcal{Q}_{k+1}\\ Q'\subseteq Q}}\prod_{j=1}^m\left(\int_{Q'}\!|f_j|^{r_j}\,\mathrm{d}x\right)^{\frac{r}{r_j}}}{\displaystyle\prod_{j=1}^m \left(\int_Q\!|f_j|^{r_j}\,\mathrm{d}x\right)^{\frac{r}{r_j}}}\leq\frac{|Q|}{2}\frac{\displaystyle\prod_{j=1}^m\left(\int_{\Omega_{k+1}\cap Q}\!|f_j|^{r_j}\,\mathrm{d}x\right)^{\frac{r}{r_j}}}{\displaystyle\prod_{j=1}^m \left(\int_Q\!|f_j|^{r_j}\,\mathrm{d}x\right)^{\frac{r}{r_j}}}\leq\frac{|Q|}{2}.
\end{align*}
Thus, defining $E_Q:=Q\backslash\Omega_{k+1}$, we have $|Q|\leq 2|E_Q|$.

To conclude that $\Sp$ is sparse, it remains to check that $(E_Q)_{Q\in\Sp}$ is pairwise disjoint. Let $Q,Q'\in\Sp$ such that $E_Q\cap E_{Q'}\neq\emptyset$. If $Q\in\mathcal{Q}_k$ and $Q'\in\mathcal{Q}_{k'}$, we have $E_Q\subseteq\Omega_k\backslash\Omega_{k+1}$ and $E_{Q'}\subseteq\Omega_{k'}\backslash\Omega_{k'+1}$. Since $(\Omega_k\backslash\Omega_{k+1})_{k\in\Z}$ is pairwise disjoint, this means that we must have $k=k'$. Since $Q\cap Q'\neq\emptyset$, it follows from maximality of $\mathcal{Q}_k$ that $Q=Q'$, as desired.

Finally, if $x\in\R^n$ and $M^{\D}_{\vec{r}}(f_1,\ldots,f_m)(x)\neq 0$, then there is a unique $k\in\Z$ such that $2^{\frac{n+1}{r}k}<M^{\D}_{\vec{r}}(f_1,\ldots,f_m)(x)\leq 2^{\frac{n+1}{r}(k+1)}$. Hence, $x\in\Omega_k\backslash\Omega_{k+1}$ and thus there is a cube $Q\in\mathcal{Q}_k$ so that $x\in Q\backslash\Omega_{k+1}=E_Q$ and
\[
M^{\D}_{\vec{r}}(f_1,\ldots,f_m)(x)\leq 2^{\frac{n+1}{r}}2^{\frac{n+1}{r}k}<2^{\frac{n+1}{r}}\prod_{j=1}^m\langle f_j\rangle_{r_j,Q}=2^{\frac{n+1}{r}}\sum_{Q'\in\Sp}\prod_{j=1}^m\langle f_j\rangle_{r_j,Q'}\chi_{E_{Q'}}(x).
\]
This proves the assertion.
\end{proof}
The following result is a reformulation of the definition of the weight class.
\begin{lemma}\label{lem:wconstchar}
Let $r_1,\ldots,r_{m+1}\in(0,\infty)$, $p_1,\ldots,p_{m+1}\in(0,\infty]$ satisfy $\frac{1}{p_j}<\frac{1}{r_j}$ for all $j\in\{1,\ldots,m+1\}$ and $\sum_{j=1}^{m+1}\frac{1}{p_j}=1$. Moreover, let $w_1,\ldots, w_{m+1}$ be weights satisfying $\prod_{j=1}^{m+1}w_j=1$ and define $v_j:=w_j^{-\frac{1}{\frac{1}{r_j}-\frac{1}{p_j}}}$. Then $\pmb{w}\in A_{\pmb{p},(\pmb{r},\infty)}$ if and only if $v_1,\ldots,v_{m+1}$ are locally integrable and there is a constant $c>0$ such that for all cubes $Q$ we have
\[
\left(\prod_{j=1}^{m+1}\langle v_j\rangle^{\frac{1}{r_j}}_{1,Q}\right)|Q|\leq c\prod_{j=1}^{m+1}v_j(Q)^{\frac{1}{p_j}}.
\]
In this case, the optimal constant $c$ in this inequality is given by $[\pmb{w}]_{\pmb{p},(\pmb{r},\infty)}$.
\end{lemma}
The following lemma allows us to deal with weighted estimates involving sparse forms.
\begin{lemma}\label{lem:woptconst}
Let $r_1,\ldots,r_{m+1}\in(0,\infty)$, $p_1,\ldots,p_{m+1}\in(0,\infty]$ satisfy $\frac{1}{p_j}<\frac{1}{r_j}$ for all $j\in\{1,\ldots,m+1\}$ and $\sum_{j=1}^{m+1}\frac{1}{p_j}=1$. Moreover, let $w_1,\ldots, w_{m+1}$ be weights satisfying $\prod_{j=1}^{m+1}w_j=1$ with $\pmb{w}\in A_{\pmb{p},(\pmb{r},\infty)}$ and define $v_j:=w_j^{-\frac{1}{\frac{1}{r_j}-\frac{1}{p_j}}}$. Let $Q$ be a cube and let $E\subseteq Q$ such that $|Q|\leq2|E|$. Then
\begin{equation}\label{eq:a2mult}
\left(\prod_{j=1}^{m+1}\langle v_j\rangle^{\frac{1}{r_j}}_{1,Q}\right)|Q|\lesssim [\pmb{w}]^{\max_{j=1,\ldots,m+1}\left\{\frac{\frac{1}{r_j}}{\frac{1}{r_j}-\frac{1}{p_j}}\right\}}_{\pmb{p},(\pmb{r},\infty)} \prod_{j=1}^{m+1}v_j(E)^{\frac{1}{p_j}}.
\end{equation}
\end{lemma}
\begin{remark}
Having Lemma~\ref{lem:wconstchar} in mind, it seems that the larger power of the weight constant in \eqref{eq:a2mult} comes from the fact that we are passing from the weighted measure of the set $Q$ to the measure of the smaller set $E$. In fact, it seems like we are only using the full weight condition $\pmb{w}\in A_{\pmb{p},(\pmb{r},\infty)}$ once and we are left with an estimate of the form
\[
\prod_{j=1}^{m+1}v_j(Q)^{\frac{1}{p_j}}\lesssim\prod_{j=1}^{m+1}v_j(E)^{\frac{1}{p_j}},
\]
where the implicit constant depends on the weights. This estimate seems to only require the weaker Fujii-Wilson $A_\infty$ condition satisfied by the weight $v_j$, but we do not pursue this further here. We refer the reader to \cite{hytonen13} where quantitative estimates involving this condition first appeared. We also point out that estimates of this type for the limited range sparse operator in the case $m=1$ have been studied in \cite{FN17,li17}. This condition has also been considered in the multilinear case in \cite{DLP15}.
\end{remark}
\begin{proof}
We set $\gamma:=\max_{j=1,\ldots,m+1}\left\{\frac{\frac{1}{r_j}}{\frac{1}{r_j}-\frac{1}{p_j}}\right\}$ and
\[
\beta_j:=\frac{1}{r_j}-\left(\frac{1}{r_j}-\frac{1}{p_j}\right)\gamma,
\]
so that $\beta_j\leq 0$ for all $j\in\{1,\ldots,m+1\}$. Thus, since $\langle v_j\rangle_{1,E}\leq2\langle v_j\rangle_{1,Q}$ by the assumptions on $E$, we have $\langle v_j\rangle_{1,Q}^{\beta_j}\leq 2^{-\beta_j}\langle v_j\rangle_{1,E}^{\beta_j}$. Then
\begin{equation}\label{eq:inmultato}
\begin{split}
\left(\prod_{j=1}^{m+1}\langle v_j\rangle^{\frac{1}{r_j}}_{1,Q}\right)|Q|&=\left(\prod_{j=1}^{m+1}\langle v_j\rangle_{1,Q}^{\frac{1}{r_j}-\frac{1}{p_j}}\right)^\gamma\left(\prod_{j=1}^{m+1}\langle v_j\rangle^{\beta_j}_{1,Q}\right)|Q|\\
&\leq[\pmb{w}]^{\gamma}_{\pmb{p},(\pmb{r},\infty)}\left(\prod_{j=1}^{m+1}\langle v_j\rangle^{\beta_j}_{1,Q}\right)|Q|\lesssim[\pmb{w}]^{\gamma}_{\pmb{p},(\pmb{r},\infty)}\left(\prod_{j=1}^{m+1}\langle v_j\rangle^{\beta_j}_{1,E}\right)|E|\\
&=[\pmb{w}]^{\gamma}_{\pmb{p},(\pmb{r},\infty)}\left(\prod_{j=1}^{m+1}v_j(E)^{\beta_j}\right)|E|^{1-\sum_{j=1}^{m+1}\beta_j}.
\end{split}
\end{equation}

Next, set $\alpha:=\sum_{j=1}^{m+1}\left(\frac{1}{r_j}-\frac{1}{p_j}\right)>0$ and $k_j:=\alpha\left(\frac{1}{r_j}-\frac{1}{p_j}\right)^{-1}$. Then
\[
\sum_{j=1}^{m+1}\frac{1}{k_j}=\frac{1}{\alpha}\sum_{j=1}^{m+1}\left(\frac{1}{r_j}-\frac{1}{p_j}\right)=1
\]
and
\[
1-\sum_{j=1}^{m+1}\beta_j=\sum_{j=1}^{m+1}\frac{1}{p_j}-\sum_{j=1}^{m+1}\frac{1}{r_j}+\gamma\sum_{j=1}^{m+1}\left(\frac{1}{r_j}-\frac{1}{p_j}\right)= (\gamma-1)\alpha
\]
so that
\[
\frac{1-\sum_{j=1}^{m+1}\beta_j}{k_j}=\left(\frac{1}{r_j}-\frac{1}{p_j}\right)(\gamma-1)=\frac{1}{p_j}-\beta_j.
\]
Thus, since $\prod_{j=1}^{m+1}v_j^{\frac{1}{r_j}-\frac{1}{p_j}}=\prod_{j=1}^{m+1}w_j=1$, it follows from H\"older's inequality that
\[
|E|^{1-\sum_{j=1}^{m+1}\beta_j}=\left(\int_E\!\prod_{j=1}^{m+1}v_j^{\frac{1}{\alpha}\left(\frac{1}{r_j}-\frac{1}{p_j}\right)}\,\mathrm{d}x\right)^{1-\sum_{j=1}^{m+1}\beta_j}\leq\prod_{j=1}^{m+1}v_j(E)^{\frac{1-\sum_{j=1}^{m+1}\beta_j}{k_j}}=\prod_{j=1}^{m+1}v_j(E)^{\frac{1}{p_j}-\beta_j}.
\]
By combining this estimate with \eqref{eq:inmultato}, we obtain \eqref{eq:a2mult}. The assertion follows.
\end{proof}
\begin{proof}[Proof of Proposition \ref{prop:mainsym}]
We set $v_j:=w_j^{-\frac{1}{\frac{1}{r_j}-\frac{1}{p_j}}}$ for $j\in\{1,\ldots,m+1\}$.

The strategy for the proof will be as follows: We will prove the equivalence of \ref{it:equ1} and \ref{it:equ2} by proving \eqref{eq:aleq1} and we will prove the equivalence of \ref{it:equ3} and \ref{it:equ4} by proving \eqref{eq:aleq2}. Then, noting that the implication \ref{it:equ3}$\Rightarrow$\ref{it:equ2} is clear, we conclude the proof by showing that \ref{it:equ1}$\Rightarrow$\ref{it:equ4} through \eqref{eq:aleq3}.

For \eqref{eq:aleq1}, for the first inequality we note that it follows from Lemma \ref{lem:dyadicmax} that it suffices to consider the estimate for $M^{\D}_{\pmb{r}}$ for a dyadic grid $\D$. First consider a finite collection $\F\subseteq\D$. Let $\lambda>0$, $f_j\in L^{p_j}(w_j^{p_j})$ and, defining $M^{\F}_{\pmb{r}}$ as $M^{\D}_{\pmb{r}}$ but with the supremum taken over all $Q\in\F$, we set
\[
\Omega^{\F}_\lambda:=\{M^{\F}_{\pmb{r}}(f_1,\ldots,f_{m+1})>\lambda\}
\]
and similarly for $\Omega^{\D}_\lambda$.

Let $\mathscr{P}$ denote the collection of those cubes $Q\in\F$ such that $\prod_{j=1}^{m+1}\langle f_j\rangle_{r_j,Q}>\lambda$ that have no dyadic ancestors in $\mathscr{F}$. Using the rule
\[
\langle h\rangle_{r,Q}=\langle h u^{-\frac{1}{r}}\rangle^u_{r,Q}\langle u\rangle^{\frac{1}{r}}_{1,Q},
\]
where $\langle h \rangle^u_{r,Q}:=\left(\frac{1}{u(Q)}\int_Q\!|h|^r u\,\mathrm{d}x\right)^{\frac{1}{r}}$,
it follows from Lemma \ref{lem:wconstchar} and the fact that $\mathscr{P}$ gives a decomposition of $\Omega^{\F}_\lambda$, that
\begin{align*}
\lambda|\Omega^{\F}_\lambda|&=\sum_{Q\in\mathscr{P}}\lambda|Q|\leq\sum_{Q\in\mathscr{P}}\left(\prod_{j=1}^{m+1}\langle f_j\rangle_{r_j,Q}\right)|Q|\\
&=\sum_{Q\in\mathscr{P}}\left(\prod_{j=1}^{m+1}\langle f_jv_j^{-\frac{1}{r_j}}\rangle^{v_j}_{r_j,Q}\langle v_j\rangle^{\frac{1}{r_j}}_{1,Q}\right)|Q|\\
&\leq[\pmb{w}]_{\pmb{p},(\pmb{r},\infty)}\sum_{Q\in\mathscr{P}}\prod_{j=1}^{m+1}\langle f_jv_j^{-\frac{1}{r_j}}\rangle^{v_j}_{r_j,Q}v_j(Q)^{\frac{1}{p_j}}\\
&\leq[\pmb{w}]_{\pmb{p},(\pmb{r},\infty)}\sum_{Q\in\mathscr{P}}\prod_{j=1}^{m+1}\left(\int_Q\! |f_j|^{p_j}v_j^{p_j\left(\frac{1}{p_j}-\frac{1}{r_j}\right)}\,\mathrm{d}x\right)^{\frac{1}{p_j}}\\
&\leq[\pmb{w}]_{\pmb{p},(\pmb{r},\infty)}\prod_{j=1}^{m+1}\|f_j\|_{L^{p_j}(w_j^{p_j})},
\end{align*}
where in the fourth step we used H\"older's inequality with $r_j\leq p_j$ and in the last step we used H\"older's inequality on the sum.

By considering an exhaustion of $\D$ of finite sets it follows from monotonicity of the measure and by taking a supremum over $\lambda>0$ that
\[
\|M^{\D}_{\pmb{r}}\|_{L^{p_1}(w_1^{p_1})\times\cdots\times L^{p_{m+1}}(w_{m+1}^{p_{m+1}})\to L^{1,\infty}}\leq[\pmb{w}]_{\pmb{p},(\pmb{r},\infty)}.
\]

For the converse inequality, fix a cube $Q$. Assuming for the moment that the $v_j$ are locally integrable, we let $0<\lambda<\prod_{j=1}^m\langle v_j\rangle^{\frac{1}{r_j}}_{1,Q}$. Setting $f_j:=v_j^{\frac{1}{r_j}}\chi_Q$, we obtain
\[
M_{\pmb{r}}(f_1,\ldots,f_{m+1})(x)\geq\prod_{j=1}^{m+1}\langle f_j\rangle_{r_j,Q}=\prod_{j=1}^{m+1}\langle v_j\rangle^{\frac{1}{r_j}}_{1,Q}>\lambda
\]
for all $x\in Q$ so that $Q\subseteq\{M_{\pmb{r}}(f_1,\ldots,f_{m+1})>\lambda\}$. Thus,
\begin{align*}
\lambda|Q|&\leq\lambda|\{M_{\pmb{r}}(f_1,\ldots,f_{m+1})>\lambda\}|\\
&\leq\|M_{\pmb{r}}\|_{L^{p_1}(w_1^{p_1})\times\cdots\times L^{p_{m+1}}(w_{m+1}^{p_{m+1}})\to L^{1,\infty}}\prod_{j=1}^{m+1}\|f_j\|_{L^{p_j}(w_j^{p_j})}\\
&=\|M_{\pmb{r}}\|_{L^{p_1}(w_1^{p_1})\times\cdots\times L^{p_{m+1}}(w_{m+1}^{p_{m+1}})\to L^{1,\infty}}\prod_{j=1}^{m+1}v_j(Q)^{\frac{1}{p_j}}.
\end{align*}
Taking a supremum over such $\lambda$, we conclude that
\begin{equation}\label{eq:lambdalim}
\left(\prod_{j=1}^{m+1}\langle v_j\rangle^{\frac{1}{r_j}}_{1,Q}\right)|Q|\leq\|M_{\pmb{r}}\|_{L^{p_1}(w_1^{p_1})\times\cdots\times L^{p_{m+1}}(w_{m+1}^{p_{m+1}})\to L^{1,\infty}}\prod_{j=1}^{m+1}v_j(Q)^{\frac{1}{p_j}}.
\end{equation}
Thus, it follows from Lemma \ref{lem:wconstchar} that
\[
[\pmb{w}]_{\pmb{p},(\pmb{r},\infty)}\leq\|M_{\pmb{r}}\|_{L^{p_1}(w_1^{p_1})\times\cdots\times L^{p_{m+1}}(w_{m+1}^{p_{m+1}})\to L^{1,\infty}},
\]
proving \eqref{eq:aleq1}. To prove our initial assumption that the $v_j$ are locally integrable, we repeat the above argument with the weights replaced by $(v_j^{-1}+\eps)^{-1}$ for $\eps>0$. As these weights are bounded, they are locally integrable. An appeal to the Monotone Convergence Theorem as $\eps\downarrow 0$ after a rearrangement of \eqref{eq:lambdalim} yields the desired conclusion.

For \eqref{eq:aleq2}, let $f_j\in L^{p_j}(w_j^{p_j})$ and let $\D$ be a dyadic grid. By Lemma \ref{lem:sparsedommax} there exists a sparse collection $\Sp\subseteq\D$ such that
\[
\|M^{\D}_{\pmb{r}}(f_1,\ldots,f_{m+1})\|_{L^1}\lesssim\|A_{\pmb{r},\Sp}(f_1,\ldots,f_{m+1})\|_{L^1}\leq\Lambda_{\pmb{r},\Sp}(f_1,\ldots,f_{m+1}).
\]
Thus, it follows from Lemma \ref{lem:dyadicmax} that
\[
\|M_{\pmb{r}}\|_{L^{p_1}(w_1^{p_1})\times\cdots\times L^{p_{m+1}}(w_{m+1}^{p_{m+1}})\to L^1}\lesssim\|\sup_{\Sp}\Lambda_{\pmb{r},\Sp}\|_{L^{p_1}(w_1^{p_1})\times\cdots\times L^{p_{m+1}}(w_{m+1}^{p_{m+1}})\to\R}.
\]
For the converse inequality, we estimate
\begin{align*}
\Lambda_{\pmb{r},\Sp}(f_1,\ldots,f_{m+1})&\leq2\sum_{Q\in\Sp}\left(\prod_{j=1}^{m+1}\langle f_j\rangle_{r_j,Q}\right)|E_Q|\\
&\leq2\sum_{Q\in\Sp}\int_{E_Q}\!M_{\pmb{r}}(f_1,\ldots,f_{m+1})\,\mathrm{d}x\\
&\leq2\|M_{\pmb{r}}(f_1,\ldots,f_m,g)\|_{L^1}.
\end{align*}
As this estimate is uniform in $\Sp$, this proves \eqref{eq:aleq2} and thus the equivalence of \ref{it:equ3} and \ref{it:equ4}.

To prove \eqref{eq:aleq3} and thus the implication \ref{it:equ1}$\Rightarrow$\ref{it:equ4}, we note that it follows from Lemma \ref{lem:woptconst} that for a sparse collection $\Sp$ in a dyadic grid $\D$ and for $\gamma=\max_{j=1,\ldots,m+1}\left\{\frac{\frac{1}{r_j}}{\frac{1}{r_j}-\frac{1}{p_j}}\right\}$ we have
\begin{align*}
\Lambda_{\pmb{r},\Sp}(f_1,\ldots,f_{m+1})&=\sum_{Q\in\Sp}\left(\prod_{j=1}^{m+1}\langle f_j\rangle_{r_j,Q}\right)|Q|\\
&=\sum_{Q\in\Sp}\left(\prod_{j=1}^{m+1}\langle f_jv_j^{-\frac{1}{r_j}}\rangle^{v_j}_{r_j,Q}\langle v_j\rangle^{\frac{1}{r_j}}_{1,Q}\right)|Q|\\
&\lesssim[\pmb{w}]^{\gamma}_{\pmb{p},(\pmb{r},\infty)}\sum_{Q\in\Sp}\prod_{j=1}^{m+1}\langle f_jv_j^{-\frac{1}{r_j}}\rangle^{v_j}_{r_j,Q}v_j(E_Q)^{\frac{1}{p_j}}\\
&\leq[\pmb{w}]^{\gamma}_{\pmb{p},(\pmb{r},\infty)}\sum_{Q\in\mathscr{P}}\prod_{j=1}^{m+1}\left(\int_{E_Q}\! M^{v_j,\D}_{r_j}(f_jv_j^{-\frac{1}{r_j}})^{p_j}v_j\,\mathrm{d}x\right)^{\frac{1}{p_j}}\\
&\leq[\pmb{w}]^{\gamma}_{\pmb{p},(\pmb{r},\infty)}\prod_{j=1}^{m+1}\|M^{v_j,\D}_{r_j}(f_jv_j^{-\frac{1}{r_j}})\|_{L^{p_j}(v_j)}\\
&\lesssim c_{\pmb{p},\pmb{r}}[\pmb{w}]^{\gamma}_{\pmb{p},(\pmb{r},\infty)}\prod_{j=1}^{m+1}\|f_j\|_{L^{p_j}(w_j^{p_j})},
\end{align*}
where in the last step we used the fact that the weighted dyadic maximal operator $M^{u,\D}_rh:=\sup_{Q\in\D}\langle h\rangle_{r,Q}^u\chi_Q$ is bounded on $L^q(u)$ for $q>r$ with constant bounded by $\left[\frac{\frac{1}{r}}{\frac{1}{r}-\frac{1}{q}}\right]^{\frac{1}{r}}$, uniformly in the weight $u$. As this estimate is uniform in the sparse collection $\Sp$, this proves \eqref{eq:aleq3}. The assertion follows.
\end{proof}

\subsection{Quantitative properties of multilinear weight classes: the $m$-tuple case}
It is sometimes convenient to emphasize this separation of the parameter $s$ from the $r_j$, as it often plays a different role from the other parameters in the proofs. The following lemma provides a way to deal with this parameter.
\begin{lemma}[Translation lemma]\label{lem:translation}
Let $r_1,\ldots,r_m\in(0,\infty)$, $s\in(0,\infty]$ and $p_1,\ldots,p_m\in(0,\infty]$ with $(\vec{r},s)\leq\vec{p}$ and let $w_1,\ldots,w_m$ be weights with $w=\prod_{j=1}^mw_j$. Then $\vec{w}\in A_{\vec{p},(\vec{r},s)}$ if and only if there are $\frac{1}{s_1},\ldots\frac{1}{s_m}$ satisfying $\frac{1}{s_j}\leq\frac{1}{p_j}$, $\sum_{j=1}^m\frac{1}{s_j}=\frac{1}{s}$, and $\vec{w}\in A_{\vec{p}(s),(\vec{r}(s),\infty)}$, where
\[
\vec{p}(s)=\left(\frac{1}{\frac{1}{p_1}-\frac{1}{s_1}},\ldots,\frac{1}{\frac{1}{p_m}-\frac{1}{s_m}}\right), \qquad\vec{r}(s)=\left(\frac{1}{\frac{1}{r_1}-\frac{1}{s_1}},\ldots,\frac{1}{\frac{1}{r_m}-\frac{1}{s_m}}\right).
\]
Moreover, in this case we have
\begin{equation}\label{eq:resclem}
[\vec{w}]_{\vec{p},(\vec{r},s)}=[\vec{w}]_{\vec{p}(s),(\vec{r}(s),\infty)}.
\end{equation}
\end{lemma}
\begin{proof}
We have
\[
\frac{1}{p(s)}:=\sum_{j=1}^m\left(\frac{1}{p_j}-\frac{1}{s_j}\right)=\frac{1}{p}-\frac{1}{s}.
\]
it remains to note that
\[
\left(\prod_{j=1}^m\langle w_j^{-1}\rangle_{\frac{1}{\frac{1}{r_j}-\frac{1}{p_j}},Q}\right)\langle w\rangle_{\frac{1}{\frac{1}{p}-\frac{1}{s}},Q}=\left(\prod_{j=1}^m\langle w_j^{-1}\rangle_{\frac{1}{\left(\frac{1}{r_j}-\frac{1}{s_j}\right)-\left(\frac{1}{p_j}-\frac{1}{s_j}\right)},Q}\right)\langle w\rangle_{p(s),Q}.
\]
Taking a supremum over all cubes $Q$ yields \eqref{eq:resclem}, proving the assertion.
\end{proof}
We point out that the choice of the $\frac{1}{s_j}$ in the lemma is not necessarily unique if $m\neq 1$. One could, for example, take $\frac{1}{s_j}=\frac{p}{p_j}\frac{1}{s}$, but a different choice will be made later in the proof of the main result. We also note that this lemma can be used even if $\frac{1}{s}=0$. In this case it can occur that some of the $\frac{1}{s_j}$ are negative, but this does not seem to cause any problems.

Having reduced to the case where $s=\infty$, the following proposition is the main result for this subsection.
\begin{proposition}\label{prop:sinf}
Let $r_1,\ldots,r_m\in(0,\infty)$, $p_1,\ldots,p_m\in(0,\infty]$ with $(\vec{r},\infty)\leq\vec{p}$ and let $w_1,\ldots,w_m$ be weights with $w=\prod_{j=1}^mw_j$. Then the following are equivalent:
\begin{enumerate}[(i)]
\item \label{it:equone1} $\vec{w}\in A_{\vec{p},(\vec{r},\infty)}$;
\item \label{it:equone2} $\|M_{\vec{r}}\|_{L^{p_1}(w_1^{p_1})\times\cdots\times L^{p_m}(w_m^{p_m})\to L^{p,\infty}(w^p)}<\infty$.
\end{enumerate}
In this case we have
\begin{equation}\label{eq:owmaxnorm}
\|M_{\vec{r}}\|_{L^{p_1}(w_1^{p_1})\times\cdots\times L^{p_m}(w_m^{p_m})\to L^{p,\infty}(w^p)}\eqsim[\vec{w}]_{\vec{p},(\vec{r},\infty)}.
\end{equation}
Moreover, if $\vec{r}<\vec{p}$, then \ref{it:equone1} and \ref{it:equone2} are equivalent to
\begin{enumerate}[(i)]\setcounter{enumi}{2}
\item \label{it:equone3} $\|M_{\vec{r}}\|_{L^{p_1}(w_1^{p_1})\times\cdots\times L^{p_m}(w_m^{p_m})\to L^p(w^p)}<\infty$
\end{enumerate}
and we have
\begin{equation}\label{eq:multibuckley}
\|M_{\vec{r}}\|_{L^{p_1}(w_1^{p_1})\times\cdots\times L^{p_m}(w_m^{p_m})\to L^p(w^p)}\lesssim c_{\vec{p},\vec{r}}[\vec{w}]^{\max_{j=1,\ldots,m}\left\{\frac{\frac{1}{r_j}}{\frac{1}{r_j}-\frac{1}{p_j}}\right\}}_{\vec{p},(\vec{r},\infty)},
\end{equation}
where the implicit constant depends on the dimension and
\[
c_{\vec{p},\vec{r}}=\prod_{j=1}^m\left[\frac{\frac{1}{r_j}}{\frac{1}{r_j}-\frac{1}{p_j}}\right]^{\frac{1}{r_j}}.
\]
Moreover, the power of the weight constant in \eqref{eq:multibuckley} is the smallest possible one.
\end{proposition}
\begin{remark}
The equivalence \eqref{eq:owmaxnorm} is also contained in the full range version in \cite[Theorem 3.3]{LOPTT09} in the case where the $p_j$ are finite, and the limited range version is proven in \cite[Proposition 21]{BM17}, but here the cases $p_j=\infty$ are only treated when $w_j=1$.

For our result here we use the interpretation that for $q=\infty$ and a weight $u$ we have $\|h\|_{L^q(u^q)}=\|h\|_{L^{q,\infty}(u^q)}=\|hu\|_{L^\infty}$.
\end{remark}
\begin{remark}
The estimate \eqref{eq:multibuckley} is a generalization of Buckley's sharp weighted bound for the Hardy-Littlewood maximal operator. It can be proven using the sparse domination we obtained in Lemma~\ref{lem:sparsedommax}, but we present an altogether different proof which generalizes an approach due to Lerner \cite{Le08}. This construction is important, as it turns out to be key for our multilinear Rubio de Francia algorithm.

In the case $r_1=\ldots=r_m=1$, the sharp bound \eqref{eq:multibuckley} recovers the sharp bound obtained by Li, Moen, and Sun in \cite{LMS14} where sparse domination techniques were used. To see this, note given weights $w_1,\ldots,w_m$ and setting $v_{\vec{w}}:=\prod_{j=1}^m w_j^{\frac{p}{p_j}}$, the multilinear weight constant they used is defined as
\begin{equation}\label{eq:frmultiwdef}
[\vec{w}]_{A_{\vec{p}}}:=\sup_{Q}\left(\frac{1}{|Q|}\int_Q\!v_{\vec{w}}\,\mathrm{d}x\right)\prod_{j=1}^m\left(\frac{1}{|Q|}\int_Q\!w_j^{1-p_j'}\,\mathrm{d}x\right)^{\frac{p}{p_j'}}.
\end{equation}
Writing $\vec{1}=(1,\ldots,1)$, the sharp result they prove is
\begin{equation}\label{eq:lmssharpbound}
\|M_{\vec{1}}\|_{L^{p_1}(w_1)\times\cdots\times L^{p_m}(w_m)\to L^p(v_{\vec{w}})}\lesssim [\vec{w}]_{A_{\vec{p}}}^{\max_{j=1,\ldots,m}\left\{\frac{p_j'}{p}\right\}}.
\end{equation}
for all $p_1,\ldots,p_m\in(1,\infty)$ and $\vec{w}\in A_{\vec{p}}$. To compare this to our result, we replace the $w_j$ by $w_j^{p_j}$ and note that $v_{\vec{w}}=\prod_{j=1}^m(w_j^{p_j})^{\frac{p}{p_j}}=w^p$, $[(w_1^{p_1},\ldots,w_m^{p_m})]_{A_{\vec{p}}}=[\vec{w}]^p_{\vec{p},(\vec{1},\infty)}.$
Thus, \eqref{eq:lmssharpbound} coincides with our bound found in \eqref{eq:multibuckley} when $\vec{r}=\vec{1}$.
\end{remark}
\begin{lemma}\label{lem:multimdomn}
Let $r_1,\ldots,r_m\in(0,\infty)$, $p_1,\ldots,p_m\in(0,\infty]$ with $\vec{r}<\vec{p}$ and let $w_1,\ldots,w_m$ be weights with $w=\prod_{j=1}^mw_j$ and $\vec{w}\in A_{\vec{p},(\vec{r},\infty)}$. Then there exist sublinear operators $N_{p_j,r_j,\vec{w}}:L^{p_j}(w_j^{p_j})\to L^{p_j}(w_j^{p_j})$ so that for any $f_j\in L^{p_j}(w_j^{p_j})$ we have
\begin{equation}\label{eq:multimdomn}
M_{\vec{r}}(f_1,\ldots,f_m)\leq[\vec{w}]^{\max_{j=1,\ldots,m}\left\{\frac{\frac{1}{r_j}}{\frac{1}{r_j}-\frac{1}{p_j}}\right\}}_{\vec{p},(\vec{r},\infty)}\prod_{j=1}^m N_{p_j,r_j,\vec{w}}(f_j).
\end{equation}
Moreover, $N_{p_j,r_j,\vec{w}}$ satisfies
\[
\|N_{p_j,r_j,\vec{w}}\|_{L^{p_j}(w_j^{p_j})\to L^{p_j}(w_j^{p_j})}\lesssim \left[\frac{\frac{1}{r_j}}{\frac{1}{r_j}-\frac{1}{p_j}}\right]^{\frac{1}{r_j}}.
\]
\end{lemma}
\begin{proof}
We first prove this result for the dyadic maximal operator $M_{\vec{r}}^{\D}$ for a dyadic grid $\D$ to obtain the appropriate operators $N^{\D}_{p_j,r_j,\vec{w}}$. Then it follows from Lemma \ref{lem:dyadicmax} that
\begin{equation}\label{ref:dyadicreduc}
M_{\vec{r}}(\vec{f})\lesssim \sum_{\alpha=1}^{3^n}\prod_{j=1}^mN^{\D_\alpha}_{p_j,r_j,\vec{w}}(f_j)\leq \prod_{j=1}^m\sum_{\alpha=1}^{3^n}N^{\D_\alpha}_{p_j,r_j,\vec{w}}(f_j).
\end{equation}
The result then follows by setting
\[
N_{p_j,r_j,\vec{w}}:=c\sum_{\alpha=1}^{3^n}N^{\D_\alpha}_{p_j,r_j,\vec{w}},
\]
where $c$ is an appropriate constant determined by the implicit constant in \eqref{ref:dyadicreduc}.

Now, fix a dyadic grid $\D$. Let $\gamma:=\max_{j=1,\ldots,m}\left\{\frac{\frac{1}{r_j}}{\frac{1}{r_j}-\frac{1}{p_j}}\right\}$, let $Q\in\D$, and set $v_j:=w_j^{-\frac{1}{\frac{1}{r_j}-\frac{1}{p_j}}}$. Since $\prod_{j=1}^m w_j^{-1}w^{\frac{\frac{1}{p_j}}{\frac{1}{p}}}=\left(\prod_{j=1}^m w_j^{-1}\right)w=1$, it follows from H\"older's inequality that
\begin{align*}
1&=\langle 1\rangle^{\gamma-1}_{\frac{1}{\sum_{j=1}^m\frac{1}{r_j}},Q}\leq\prod_{j=1}^m\langle w_j^{-1}w^{\frac{\frac{1}{p_j}}{\frac{1}{p}}}\rangle^{\gamma-1}_{r_j,Q}=\prod_{j=1}^m\langle w_j^{-1}w^{\frac{\frac{1}{p_j}}{\frac{1}{p}}}\rangle^{\gamma-\frac{\frac{1}{r_j}}{\frac{1}{r_j}-\frac{1}{p_j}}}_{r_j,Q}\langle w_j^{-1}w^{\frac{\frac{1}{p_j}}{\frac{1}{p}}}\rangle^{\frac{\frac{1}{p_j}}{\frac{1}{r_j}-\frac{1}{p_j}}}_{r_j,Q}\\
&\leq\prod_{j=1}^m\left(\langle w_j^{-1}\rangle_{\frac{1}{\frac{1}{r_j}-\frac{1}{p_j}}}\langle w^{\frac{\frac{1}{p_j}}{\frac{1}{p}}}\rangle_{p_j,Q}\right)^{\gamma-\frac{\frac{1}{r_j}}{\frac{1}{r_j}-\frac{1}{p_j}}}\langle w_j^{-1}w^{\frac{\frac{1}{p_j}}{\frac{1}{p}}}\rangle^{\frac{\frac{1}{p_j}}{\frac{1}{r_j}-\frac{1}{p_j}}}_{r_j,Q}\\
&=\Big(\prod_{j=1}^m \left(\langle v_j\rangle^{\frac{1}{r_j}-\frac{1}{p_j}}_{1,Q}\langle w^p\rangle^{\frac{1}{p_j}}_{1,Q}\right)^{\gamma-\frac{\frac{1}{r_j}}{\frac{1}{r_j}-\frac{1}{p_j}}}\Big)\prod_{j=1}^m \langle w_j^{-r_j}w^{\frac{\frac{1}{p_j}}{\frac{1}{p}}r_j}\rangle^{\frac{\frac{1}{p_j}\frac{1}{r_j}}{\frac{1}{r_j}-\frac{1}{p_j}}}_{1,Q}.
\end{align*}
This implies that
\begin{align*}
\prod_{j=1}^m\langle v_j\rangle_{1,Q}^{\frac{1}{r_j}}&\leq\frac{[\vec{w}]^\gamma_{\vec{p},(\vec{r},\infty)}}{\Big(\displaystyle\prod_{j=1}^m\langle v_j\rangle^{\left(\frac{1}{r_j}-\frac{1}{p_j}\right)\gamma-\frac{1}{r_j}}_{1,Q}\Big)\langle w\rangle_{p,Q}^\gamma}\\ &=\frac{[\vec{w}]^\gamma_{\vec{p},(\vec{r},\infty)}}{\displaystyle\prod_{j=1}^m \left(\langle v_j\rangle^{\frac{1}{r_j}-\frac{1}{p_j}}_{1,Q}\langle w^p\rangle^{\frac{1}{p_j}}_{1,Q}\right)^{\gamma-\frac{\frac{1}{r_j}}{\frac{1}{r_j}-\frac{1}{p_j}}}} \prod_{j=1}^m\left(\frac{1}{\langle w^p\rangle_{1,Q}}\right)^{\frac{\frac{1}{p_j}\frac{1}{r_j}}{\frac{1}{r_j}-\frac{1}{p_j}}}\\
&\leq[\vec{w}]^\gamma_{\vec{p},(\vec{r},\infty)}\prod_{j=1}^m\bigg(\frac{\langle w_j^{-r_j}w^{\frac{\frac{1}{p_j}}{\frac{1}{p}}r_j}\rangle_{1,Q}}{\langle w^p\rangle_{1,Q}}\bigg)^{\frac{\frac{1}{p_j}\frac{1}{r_j}}{\frac{1}{r_j}-\frac{1}{p_j}}}.
\end{align*}
Thus, for $f_j\in L^{p_j}(w_j^{p_j})$ and any $x\in Q$, we have
\begin{equation}\label{eq:splitmultibuck}
\begin{split}
\prod_{j=1}^m\langle f_j\rangle_{r_j,Q}&=\prod_{j=1}^m\langle f_j v_j^{-\frac{1}{r_j}}\rangle^{v_j}_{r_j,Q}\langle v_j\rangle^{\frac{1}{r_j}}_{1,Q}\\
&\leq[\vec{w}]^\gamma_{\vec{p},(\vec{r},\infty)}\prod_{j=1}^m\Bigg(\frac{\inf_{y\in Q}M^{v_j,\D}_{r_j}(f_jv_j^{-\frac{1}{r_j}})(y)^{\frac{\frac{1}{r_j}-\frac{1}{p_j}}{\frac{1}{p_j}\frac{1}{r_j}}}\langle w_j^{-r_j}w^{\frac{\frac{1}{p_j}}{\frac{1}{p}}r_j}\rangle_{1,Q}}{\langle w^p\rangle_{1,Q}}\Bigg)^{\frac{\frac{1}{p_j}\frac{1}{r_j}}{\frac{1}{r_j}-\frac{1}{p_j}}}\\
&\leq[\vec{w}]^\gamma_{\vec{p},(\vec{r},\infty)}\prod_{j=1}^m M^{w^p,\D}_{\frac{\frac{1}{r_j}-\frac{1}{p_j}}{\frac{1}{p_j}\frac{1}{r_j}}}(M^{v_j,\D}_{r_j}(f_jv_j^{-\frac{1}{r_j}})v_j^{\frac{1}{p_j}}w^{-\frac{\frac{1}{p_j}}{\frac{1}{p}}})(x).
\end{split}
\end{equation}
Setting
\[
N^{\D}_{p_j,r_j,\vec{w}}(f_j):=M^{w^p,\D}_{\frac{\frac{1}{r_j}-\frac{1}{p_j}}{\frac{1}{p_j}\frac{1}{r_j}}}(M^{v_j,\D}_{r_j}(f_jv_j^{-\frac{1}{r_j}})v_j^{\frac{1}{p_j}}w^{-\frac{\frac{1}{p_j}}{\frac{1}{p}}})w^{\frac{\frac{1}{p_j}}{\frac{1}{p}}}w_j^{-1}
\]
and by taking a supremum over all $Q$ containing $x$ in \eqref{eq:splitmultibuck} we have proven \eqref{eq:multimdomn} in the dyadic case. We remark here that in the case that $\frac{1}{p_j}=0$, we use the interpretation
\[
N^{\D}_{\infty,r_j,\vec{w}}(f_j)=\|M^{v_j,\D}_{r_j}(f_jv_j^{-\frac{1}{r_j}})\|_{L^\infty}w_j^{-1}.
\]

Noting that
\[
\|M^{u,\D}_{\frac{\frac{1}{r}-\frac{1}{q}}{\frac{1}{q}\frac{1}{r}}}(h)\|_{L^{q}(u)}\lesssim\left(\frac{q}{r}\right)^{\frac{\frac{1}{q}\frac{1}{r}}{\frac{1}{r}-\frac{1}{q}}}\|h\|_{L^q(u)}=e^{\frac{\log q-\log r}{q-r}}\|h\|_{L^q(u)}\leq e^{\frac{1}{r}}\|h\|_{L^q(u)},
\]
for the case $\frac{1}{p_j}>0$, we compute
\begin{align*}
\|N^{\D}_{p_j,r_j,\vec{w}}(f_j)\|_{L^{p_j}(w_j^{p_j})}&=\|M^{w^p,\D}_{\frac{\frac{1}{r_j}-\frac{1}{p_j}}{\frac{1}{p_j}\frac{1}{r_j}}}(M^{v_j,\D}_{r_j}(f_jv_j^{-\frac{1}{r_j}})v_j^{\frac{1}{p_j}}w^{-\frac{p}{p_j}})\|_{L^{p_j}(w^p)}\\
&\lesssim\|M^{v_j,\D}_{r_j}(f_jv_j^{-\frac{1}{r_j}})v_j^{\frac{1}{p_j}}w^{-\frac{p}{p_j}}\|_{L^{p_j}(w^p)}\\
&=\|M^{v_j,\D}_{r_j}(f_jv_j^{-\frac{1}{r_j}})\|_{L^{p_j}(v_j)}\\
&\lesssim \left[\frac{\frac{1}{r_j}}{\frac{1}{r_j}-\frac{1}{p_j}}\right]^{\frac{1}{r_j}}\|f_jv_j^{-\frac{1}{r_j}}\|_{L^{p_j}(v_j)}\\
&=\left[\frac{\frac{1}{r_j}}{\frac{1}{r_j}-\frac{1}{p_j}}\right]^{\frac{1}{r_j}}\|f_j\|_{L^{p_j}(w_j^{p_j})},
\end{align*}
and for the case $\frac{1}{p_j}=0$, we compute
\[
\|N^{\D}_{\infty,r_j,\vec{w}}(f_j)w_j\|_{L^\infty}=\|M^{v_j,\D}_{r_j}(f_jv_j^{-\frac{1}{r_j}})\|_{L^\infty}\leq\|f_jv_j^{-\frac{1}{r_j}}\|_{L^\infty}=\|f_jw_j\|_{L^\infty}.
\]
The assertion follows.
\end{proof}
\begin{proof}[Proof of Proposition \ref{prop:sinf}]
We will prove the equivalence of \ref{it:equone1} and \ref{it:equone2} by proving \eqref{eq:owmaxnorm}.

For $``\lesssim"$, we note that it follows from Lemma \ref{lem:dyadicmax} that it suffices to prove the estimate for $M_{\vec{r}}^{\D}$ for a fixed dyadic grid $\D$. Note that by H\"older's inequality we have $\langle f_j\rangle_{r_j,Q}\leq\langle f_jw_j\rangle_{p_j,Q}\langle w_j^{-1}\rangle_{\frac{1}{\frac{1}{r_j}-\frac{1}{p_j}},Q}$ for a cube $Q$, so that
\[
\prod_{j=1}^m\langle f_j\rangle_{r_j,Q}\leq[\vec{w}]_{\vec{p},(\vec{r},\infty)}\langle w\rangle^{-1}_{p,Q}\prod_{j=1}^m\langle f_jw_j\rangle_{p_j,Q}=[\vec{w}]_{\vec{p},(\vec{r},\infty)}\prod_{j=1}^m\langle f_jw_jw^{-\frac{p}{p_j}}\rangle^{w^p}_{p_j,Q}.
\]
Thus, by H\"older's inequality for weak type Lebesgue spaces, we have
\begin{align*}
\|M^{\D}_{\vec{r}}(f_1,\ldots,f_m)\|_{L^{p,\infty}(w^p)}&\leq[\vec{w}]_{\vec{p},(\vec{r},\infty)}\|\prod_{j=1}^m M_{p_j}^{w^p,\D}(f_jw_jw^{-\frac{p}{p_j}})\|_{L^{p,\infty}(w^p)}\\
&\lesssim[\vec{w}]_{\vec{p},(\vec{r},\infty)}\prod_{j=1}^m\|M_{p_j}^{w^p,\D}(f_jw_jw^{-\frac{p}{p_j}})\|_{L^{p_j,\infty}(w^p)}\\
&\lesssim[\vec{w}]_{\vec{p},(\vec{r},\infty)}\prod_{j=1}^m\|f_j\|_{L^{p_j}(w_j^{p_j})},
\end{align*}
where we used the fact that the weighted dyadic maximal operator $M^{u,\D}_q$ is bounded $L^q(u)\to L^{q,\infty}(u)$ with constant uniform in $q$ and the weight $u$. Thus, we have shown that
\[
\|M_{\vec{r}}\|_{L^{p_1}(w_1^{p_1})\times\cdots\times L^{p_m}(w_m^{p_m})\to L^{p,\infty}(w^p)}\lesssim[\vec{w}]_{\vec{p},(\vec{r},\infty)}.
\]

For the converse inequality, fix a cube $Q$ and let $f_j\in L^{p_j}(w_j^{p_j})$. Letting $0<\lambda<\prod_{j=1}^m\langle f_j\rangle_{r_j,Q}$, we have
\[
M_{\vec{r}}(f_1,\ldots,f_m)(x)\geq\prod_{j=1}^m\langle f_j\rangle_{r_j,Q}>\lambda
\]
for all $x\in Q$ so that $Q\subseteq\{M_{\vec{r}}(f_1,\ldots,f_m)>\lambda\}$. Hence,
\begin{align*}
\lambda \langle w\rangle_{p,Q}&\leq|Q|^{-\frac{1}{p}}\lambda (w^p(\{M_{\vec{r}}(f_1,\ldots,f_m)>\lambda\}))^{\frac{1}{p}}\\
&\leq\|M_{\vec{r}}\|_{L^{p_1}(w_1^{p_1})\times\cdots\times L^{p_m}(w_m^{p_m})\to L^{p,\infty}(w^p)}\prod_{j=1}^m|Q|^{-\frac{1}{p_j}}\|f_j\|_{L^{p_j}(w^{p_j})}.
\end{align*}
Taking a supremum over such $\lambda$ and by replacing $f_j$ with $\chi_Q f_j$, we conclude that
\begin{equation}\label{eq:fjmreq}
\left(\prod_{j=1}^m\langle f_j\rangle_{r_j,Q}\right)\langle w\rangle_{p,Q}\leq\|M_{\vec{r}}\|_{L^{p_1}(w_1^{p_1})\times\cdots\times L^{p_m}(w_m^{p_m})\to L^{p,\infty}(w^p)}\prod_{j=1}^m\langle f_jw_j\rangle_{p_j,Q}.
\end{equation}
Now set $f_j=w_j^{-\frac{\frac{1}{r_j}}{\frac{1}{r_j}-\frac{1}{p_j}}}$ and assume for the moment that $f_j^{r_j}=f_j^{p_j}w_j^{p_j}=w_j^{-\frac{1}{\frac{1}{r_j}-\frac{1}{p_j}}}$ is locally integrable. Then the product on the right-hand side of \eqref{eq:fjmreq} is positive and finite so that we may take it to the left-hand side. This yields
\begin{equation}\label{eq:fjmreq2}
\left(\prod_{j=1}^m\langle w_j^{-1}\rangle_{\frac{1}{\frac{1}{r_j}-\frac{1}{p_j}},Q}\right)\langle w\rangle_{p,Q}\leq\|M_{\vec{r}}\|_{L^{p_1}(w_1^{p_1})\times\cdots\times L^{p_m}(w_m^{p_m})\to L^{p,\infty}(w^p)}
\end{equation}
and taking a supremum over all cubes $Q$ yields \eqref{eq:owmaxnorm}. To prove that $w_j^{-\frac{1}{\frac{1}{r_j}-\frac{1}{p_j}}}$ is indeed locally integrable, we choose $f_j$ such that $f_j^{p_j}w_j^{p_j}=(w_j^{\frac{1}{\frac{1}{r_j}-\frac{1}{p_j}}}+\eps)^{-1}$ for $\eps>0$, the latter expression being bounded and thus locally integrable. Again taking the product on the right-hand side of \eqref{eq:fjmreq} to the left, an appeal to the Monotone Convergence Theorem as $\eps\downarrow 0$ yields \eqref{eq:fjmreq2}. The assertion follows.

Since the implication \ref{it:equone3}$\Rightarrow$\ref{it:equone2} is clear, we may finish the proof of the equivalences by showing \ref{it:equone1}$\Rightarrow$\ref{it:equone3} through \eqref{eq:multibuckley}.

By Lemma \ref{lem:multimdomn}, it follows from H\"older's inequality that
\begin{align*}
\|M_{\vec{r}}(f_1,\ldots,f_m)\|_{L^p(w^p)}&\leq[\vec{w}]^{\max_{j=1,\ldots,m}\left\{\frac{\frac{1}{r_j}}{\frac{1}{r_j}-\frac{1}{p_j}}\right\}}_{\vec{p},(\vec{r},\infty)}\prod_{j=1}^m \|N_{p_j,r_j,\vec{w}}f_j\|_{L^{p_j}(w_j^{p_j})}\\
&\lesssim c_{\vec{p},\vec{r}}[\vec{w}]^{\max_{j=1,\ldots,m}\left\{\frac{\frac{1}{r_j}}{\frac{1}{r_j}-\frac{1}{p_j}}\right\}}_{\vec{p},(\vec{r},\infty)}\prod_{j=1}^m\|f_j\|_{L^{p_j}(w_j^{p_j})},
\end{align*}
as desired.

Finally, we prove optimality of \eqref{eq:multibuckley}. Let $\alpha\geq 0$ denote the smallest possible constant in the estimate
\[
\|M_{\vec{r}}(f_1,\ldots,f_m)\|_{L^p(w^p)}\lesssim[\vec{w}]^\alpha_{\vec{p},(\vec{r},\infty)}\prod_{j=1}^m\|f_j\|_{L^{p_j}(w_j^{p_j})}.
\]
We have shown that $\alpha\leq\max_{j=1,\ldots,m}\left\{\frac{\frac{1}{r_j}}{\frac{1}{r_j}-\frac{1}{p_j}}\right\}$ and it remains to prove the lower bound. We assume that we are in dimension $n=1$, the general case following mutatis mutandis. Moreover, we assume without loss of generality that the maximum $\max_{j=1,\ldots,m}\left\{\frac{\frac{1}{r_j}}{\frac{1}{r_j}-\frac{1}{p_j}}\right\}$ is attained for $j=1$, the other cases following similarly by permuting the indices. For $0<\eps<1$ we define
\begin{align*}
w_1(x):=|x|^{(1-\eps)\left(\frac{1}{r_1}-\frac{1}{p_1}\right)},&\qquad w_j(x):=1\quad\text{for $j\in\{2,\ldots,m\}$},\\
f_1(x):=|x|^{-\frac{1-\eps}{r_1}}\chi_{(0,1)}(x),&\qquad f_j(x):=|x|^{-\frac{1-\eps}{p_j}}\chi_{(0,1)}(x)\quad\text{for $j\in\{2,\ldots,m\}$}.
\end{align*}
Then, by H\"older's inequality and a computation, we have
\[
[\vec{w}]_{\vec{p},(\vec{r},\infty)}\leq[w_1]_{p_1,(r_1,\infty)}\eqsim\eps^{\frac{1}{p_1}-\frac{1}{r_1}}.
\]
Moreover, one computes
\[
\prod_{j=1}^m\|f_j\|_{L^{p_j}(w_j^{p_j})}=\eps^{-\frac{1}{p}}
\]
and
\[
\prod_{j=1}^m\langle f_j\rangle_{r_j,[-|x|,|x|]}\gtrsim\eps^{-\frac{1}{r_1}}f_1(x)\prod_{j=2}^m\left[\frac{\frac{1}{r_j}}{\frac{1}{r_j}-(1-\eps)\frac{1}{p_j}}\right]^{\frac{1}{r_j}}f_j(x).
\]
Setting $f(x):=\prod_{j=1}^m f_j(x)w_j(x)=|x|^{-\frac{1-\eps}{p}}\chi_{(0,1)}(x)$, we find that
\[
\|M_{\vec{r}}(f_1,\ldots,f_m)\|_{L^p(w^p)}\gtrsim\eps^{-\frac{1}{r_1}}\|f\|_{L^p}=\eps^{-\frac{1}{r_1}-\frac{1}{p}}
\]
and
\[
\|M_{\vec{r}}(f_1,\ldots,f_m)\|_{L^p(w^p)}\lesssim[\vec{w}]^\alpha_{\vec{p},(\vec{r},\infty)}\prod_{j=1}^m\|f_j\|_{L^{p_j}(w_j^{p_j})}\lesssim\eps^{\alpha\left(\frac{1}{p_1}-\frac{1}{r_1}\right)-\frac{1}{p}}
\]
Letting $\eps\downarrow 0$ shows that we must have $\alpha\left(\frac{1}{p_1}-\frac{1}{r_1}\right)-\frac{1}{p}-(-\frac{1}{r_1}-\frac{1}{p})\leq 0$, i.e.,
\[
\alpha\geq\frac{\frac{1}{r_1}}{\frac{1}{r_1}-\frac{1}{p_1}}=\max_{j=1,\ldots,m}\left\{\frac{\frac{1}{r_j}}{\frac{1}{r_j}-\frac{1}{p_j}}\right\}.
\]
Thus, we have $\alpha=\max_{j=1,\ldots,m}\left\{\frac{\frac{1}{r_j}}{\frac{1}{r_j}-\frac{1}{p_j}}\right\}$ and the assertion follows.
\end{proof}
\begin{remark}
In the unweighted case we actually have an equivalence $\|M_{\vec{r}}\|_{L^{p_1}\times\cdots\times L^{p_m}\to L^p}\eqsim c_{\vec{p},\vec{r}}$, which follows from a similar calculation as above, with $f_j(x):=|x|^{-\frac{1-\eps}{p_j}}\chi_{(0,1)}(x)$ for all $j\in\{1,\ldots,m\}$.
\end{remark}
\section{Proof of the main result}\label{sec:proof}
The proof of the main theorem essentially follows from the theorem below. In this theorem we deal with $m+1$-tuples as well as $m$-tuples of the same parameters, which can be notationally confusing. To circumvent this problem, we shall use the earlier established convention that for $m+1$ parameters $\alpha_1,\ldots,\alpha_{m+1}$ we shall use the boldface notation $\pmb{\alpha}=(\alpha_1,\ldots,\alpha_{m+1})$ for $m+1$-tuples while we will use the arrow notation $\vec{\alpha}=(\alpha_1,\ldots,\alpha_m)$ for $m$-tuples, see also Section~\ref{sec:twotwo}.

We again point out that even though this result is formulated for the Banach range $\frac{1}{p}\leq 1$, it can be used to obtain results in the range including the cases $\frac{1}{p}>1$, see also Remark \ref{rem:qbvsb} and the proof of Theorem \ref{thm:qmlre}.
\begin{theorem}\label{thm:main}
Let $\frac{1}{r_1},\ldots,\frac{1}{r_{m+1}}\in(0,1]$ and suppose we are given $\frac{1}{p_1},\ldots,\frac{1}{p_{m+1}}\in[0,1]$ satisfying $\frac{1}{p_j}<\frac{1}{r_j}$ for all $j\in\{1,\ldots,m+1\}$ and $\sum_{j=1}^{m+1}\frac{1}{p_j}=1$. Assume moreover that we are given weights $w_1,\ldots w_{m+1}$ satisfying $\prod_{j=1}^{m+1}w_j=1$ and $\pmb{w} \in A_{\pmb{p},(\pmb{r},\infty)}$.

Suppose we are given functions $f_j\in L^{p_j}(w_j^{p_j})$ and $\frac{1}{q_1},\ldots,\frac{1}{q_{m+1}}\in[0,1]$ satisfying $\frac{1}{q_j}\leq\frac{1}{r_j}$ and $\sum_{j=1}^{m+1}\frac{1}{q_j}=1$. Then there are weights $W_1,\ldots, W_{m+1}$ satisfying $\prod_{j=1}^{m+1}W_j=1$ and $\pmb{W}\in A_{\pmb{q},(\pmb{r},\infty)}$ such that
\begin{equation}\label{eq:mainprop1}
\prod_{j=1}^{m+1}\|f_j\|_{L^{q_j}(W_j^{q_j})}\leq 2^{m^2}\prod_{j=1}^{m+1}\|f_j\|_{L^{p_j}(w_j^{p_j})}
\end{equation}
and
\begin{equation}\label{eq:mainprop2}
[\pmb{W}]_{\pmb{q},(\pmb{r},\infty)}\leq C_{\pmb{p},\pmb{q},\pmb{r}}[\pmb{w}]^{\max_{j=1,\ldots,m+1} \left\{\frac{\frac{1}{r_j}-\frac{1}{q_j}}{\frac{1}{r_j}-\frac{1}{p_j}}\right\}}_{\pmb{p},(\pmb{r},\infty)}.
\end{equation}
\end{theorem}
The proof of this theorem relies on a multilinear generalization of the Rubio de Francia algorithm.
\begin{lemma}[Multilinear Rubio de Francia algorithm]\label{lem:multirubalg}
Let $r_1,\ldots,r_m,p_1,\ldots,p_m\in (0,\infty)$ with $\vec{r}<\vec{p}$. Then for each $\vec{w}\in A_{\vec{p},(\vec{r},\infty)}$ there exist operators $R_{p_j,r_j,\vec{w}}:L^{p_j}(w_j^{p_j})\to L^{p_j}(w_j^{p_j})$ satisfying
\begin{enumerate}[(i)]
\item \label{it:ru1}
$|f_j|\leq R_{p_j,r_j,\vec{w}}f_j$;
\item \label{it:ru2} $\|R_{p_j,r_j,\vec{w}}f_j\|_{L^{p_j}(w_j^{p_j})}\leq2\|f_j\|_{L^{p_j}(w_j^{p_j})}$;
\item \label{it:ru3} $\displaystyle\prod_{j=1}^m\langle R_{p_j,r_j,\vec{w}}f_j\rangle_{r_j,Q}\lesssim c_{\vec{p},\vec{r}}[\vec{w}]^{\max_{j=1,\ldots,m}\left\{\frac{\frac{1}{r_j}}{\frac{1}{r_j}-\frac{1}{p_j}}\right\}}_{\vec{p},(\vec{r},\infty)} \inf_{y\in Q}\prod_{j=1}^mR_{p_j,r_j,\vec{w}}f_j(y)$ for all cubes $Q$, where the implicit constant depends on the dimension and
\[
c_{\vec{p},\vec{r}}=\prod_{j=1}^m\left[\frac{\frac{1}{r_j}}{\frac{1}{r_j}-\frac{1}{p_j}}\right]^{\frac{1}{r_j}}.
\]
\end{enumerate}
\end{lemma}
\begin{proof}
Letting $N_{p_j,r_j,\vec{w}}$ be as in Lemma \ref{lem:multimdomn}, we define
\[
R_{p_j,r_j,\vec{w}}f_j:=\sum_{k=0}^\infty\frac{N^k_{p_j,r_j,\vec{w}}(f_j)}{2^k\|N_{p_j,r_j,\vec{w}}\|^k_{L^{p_j}(w_j^{p_j})\to L^{p_j}(w_j^{p_j})}},
\]
where $N^0_{p_j,r_j,\vec{w}}(f_j):=|f_j|$ and $N^{k}_{p_j,r_j,\vec{w}}(f_j):=N_{p_j,r_j,\vec{w}}(N^{k-1}_{p_j,r_j,\vec{w}}(f_j))$.

To prove property \ref{it:ru1}, it suffices to note that the $k=0$ term in the sum is equal to $|f_j|$.

For \ref{it:ru2} we have
\begin{align*}
\|R_{p_j,r_j,\vec{w}}f_j\|_{L^{p_j}(w_j^{p_j})}&\leq\sum_{k=0}^\infty\frac{\|N^k_{p_j,r_j,\vec{w}}(f_j)\|_{L^{p_j}(w_j^{p_j})}}{2^k\|N_{p_j,r_j,\vec{w}}\|^k_{L^{p_j}(w_j^{p_j})\to L^{p_j}(w_j^{p_j})}}\\
&\leq\sum_{k=0}^\infty\frac{\|f_j\|_{L^{p_j}(w_j^{p_j})}}{2^k}=2\|f_j\|_{L^{p_j}(w_j^{p_j})}.
\end{align*}

To prove \ref{it:ru3}, we first note that
\begin{align*}
N_{p_j,r_j,\vec{w}}(R_{p_j,r_j,\vec{w}}f_j)&\leq\sum_{k=0}^\infty\frac{N^{k+1}_{p_j,r_j,\vec{w}}(f_j)}{2^k\|N_{p_j,r_j,\vec{w}}\|^k_{L^{p_j}(w_j^{p_j})\to L^{p_j}(w_j^{p_j})}}\\
&\leq2\|N_{p_j,r_j,\vec{w}}\|_{L^{p_j}(w_j^{p_j})\to L^{p_j}(w_j^{p_j})}R_{p_j,r_j,\vec{w}}f_j.
\end{align*}
Thus, it follows from Lemma \ref{lem:multimdomn} that
\begin{align*}
M_{\vec{r}}(R_{p_1,r_1,\vec{w}}f_1,\ldots,R_{p_m,r_m,\vec{w}}f_m)&\leq [\vec{w}]^{\max_{j=1,\ldots,m}\left\{\frac{\frac{1}{r_j}}{\frac{1}{r_j}-\frac{1}{p_j}}\right\}}_{\vec{p},(\vec{r},\infty)} \prod_{j=1}^m N_{p_j,r_j,\vec{w}}(R_{p_j,r_j,\vec{w}}f_j)\\
&\lesssim 2^mc_{\vec{p},\vec{r}}[\vec{w}]^{\max_{j=1,\ldots,m}\left\{\frac{\frac{1}{r_j}}{\frac{1}{r_j}-\frac{1}{p_j}}\right\}}_{\vec{p},(\vec{r},\infty)} \prod_{j=1}^mR_{p_j,r_j,\vec{w}}f_j,
\end{align*}
as desired. The assertion follows.
\end{proof}
\begin{proof}[Proof of Theorem \ref{thm:main}]
The proof will consist of two steps. In the first step we prove the result for very specific $\pmb{q}$. In the second step we iterate the first step to obtain the desired result.

\emph{Step 1.} In this step we assume that there is some $j_0\in\{1,\ldots,m+1\}$ such that
\[
\frac{1}{p_{j_0}}<\frac{1}{q_{j_0}},\qquad \frac{1}{p_j}\geq\frac{1}{q_j}\quad\text{for $j\neq j_0$.}
\]
Since none of the statements in the formulation of the proposition depend on the order of the indices, we may assume without loss of generality that $j_0=m+1$. More precisely, we can let $\pi\in S_{m+1}$ be the transposition given by $\pi(j)=j$ for $j\neq j_0,m+1$ and $\pi(j_0)=m+1$, $\pi(m+1)=j_0$. Replacing the index $j$ by $\pi(j)$ everywhere then indeed allows us to reduce to the case $j_0=m+1$.

We define $\frac{1}{s}:=1-\frac{1}{r_{m+1}}\geq0$, $\frac{1}{p}:=1-\frac{1}{p_{m+1}}>0$, $\frac{1}{q}:=1-\frac{1}{q_{m+1}}\geq 0$, and $w:=w_{m+1}^{-1}$ so that $w=\prod_{j=1}^m w_j$. For an $m+1$-tuple $(\alpha_1,\ldots,\alpha_{m+1})$ we will use the notation $\vec{\alpha}=(\alpha_1,\ldots,\alpha_m)$ so that the arrow notation will always refer to an $m$-tuple. Thus, we have now reduced the problem to proving that there exist $m$ weights $\vec{W}\in A_{\vec{q},(\vec{r},s)}$ such that  $f_j\in L^{q_j}(W_j^{q_j})$, $f_{m+1}\in L^{q'}(W^{-q'})$, where $W:=\prod_{j=1}^m W_j$, with
\begin{equation}\label{eq:bigwest}
\left(\prod_{j=1}^m\|f_j\|_{L^{q_j}(W_j^{q_j})}\right)\|f_{m+1}\|_{L^{q'}(W^{-q'})}\leq 2^m\left(\prod_{j=1}^m\|f_j\|_{L^{p_j}(w_j^{p_j})}\right)\|f_{m+1}\|_{L^{p'}(w^{-p'})},
\end{equation}
and
\begin{equation}\label{eq:ru4}
[\vec{W}]_{\vec{q},(\vec{r},s)}\leq C_{\vec{p},\vec{q},\vec{r},s}[\vec{w}]^{\max_{j=1,\ldots,m} \left\{\frac{\frac{1}{r_j}-\frac{1}{q_j}}{\frac{1}{r_j}-\frac{1}{p_j}}\right\}}_{\vec{p},(\vec{r},s)}.
\end{equation}
Indeed, the result then follows by setting $W_{m+1}:=W^{-1}$ and by noting that
\[
[\pmb{W}]_{\pmb{q},(\pmb{r},\infty)}=[\vec{W}]_{\vec{q},(\vec{r},s)},\qquad[\pmb{w}]_{\pmb{p},(\pmb{r},\infty)}=[\vec{w}]_{\vec{p},(\vec{r},s)}.
\]

The construction of the $m$ weights $W_1,\ldots, W_m$ relies on the multilinear Rubio de Francia algorithm as well as a clever usage of the translation lemma to deal with the parameter $s$. Setting
\[
\frac{1}{s_j}:=\frac{\left(\frac{1}{p}-\frac{1}{s}\right)\frac{1}{q_j}-\left(\frac{1}{q}-\frac{1}{s}\right)\frac{1}{p_j}}{\frac{1}{p}-\frac{1}{q}},
\]
we have
\[
\frac{1}{s_j}\leq\frac{\left(\frac{1}{p}-\frac{1}{s}\right)\frac{1}{q_j}-\left(\frac{1}{q}-\frac{1}{s}\right)\frac{1}{q_j}}{\frac{1}{p}-\frac{1}{q}}=\frac{1}{q_j}
\]
with equality if and only if $\frac{1}{q_j}=\frac{1}{p_j}$ and so that $\frac{1}{s_j}\leq\frac{1}{q_j}\leq\frac{1}{p_j}$, and
\[
\sum_{j=1}^m\frac{1}{s_j}=\frac{\left(\frac{1}{p}-\frac{1}{s}\right)\frac{1}{q}-\left(\frac{1}{q}-\frac{1}{s}\right)\frac{1}{p}}{\frac{1}{p}-\frac{1}{q}}=\frac{1}{s}.
\]
We set
\[
\frac{1}{p_j(s)}:=\frac{1}{p_j}-\frac{1}{s_j},\qquad\frac{1}{q_j(s)}:=\frac{1}{q_j}-\frac{1}{s_j},\qquad
\frac{1}{r_j(s)}:=\frac{1}{r_j}-\frac{1}{s_j}
\]
and $\frac{1}{p(s)}:=\sum_{j=1}^m\frac{1}{p_j(s)}=\frac{1}{p}-\frac{1}{s}$, $\vec{p}(s):=(p_1(s),\ldots,p_m(s))$, and similarly for $\frac{1}{q(s)}$, $\vec{q}(s)$, and $\vec{r}(s)$.

We emphasize here that $\frac{1}{p_j(s)}=0$ if and only if $\frac{1}{p_j}=\frac{1}{q_j}$ and we encourage the reader to verify that the remaining steps in this proof remain valid in this particular case.

We may compute
\begin{equation}\label{eq:joeq}
\frac{1}{p_j}-\frac{1}{q_j}=\frac{\frac{1}{p(s)}-\frac{1}{q(s)}}{\frac{1}{p(s)}}\frac{1}{p_j(s)},\qquad \frac{1}{q_j(s)}=\frac{\frac{1}{q(s)}}{\frac{1}{p(s)}}\frac{1}{p_j(s)}.
\end{equation}

We set $g_j:=|f_j|^{\frac{\frac{1}{p_j(s)}}{\frac{1}{p_j}}}w_j^{-\frac{\frac{1}{s_j}}{\frac{1}{p_j}}}$ so that
\[
\|g_j\|_{L^{p_j(s)}(w_j^{p_j(s)})}=\|f_j\|^{\frac{\frac{1}{p_j(s)}}{\frac{1}{p_j}}}_{L^{p_j}(w_j^{p_j})}
\]
and, using the notation from Lemma \ref{lem:multirubalg}, we set
\[
W_j:=(R_{p_j(s),r_j(s),\vec{w}}(g_j))^{-\frac{\frac{1}{p(s)}-\frac{1}{q(s)}}{\frac{1}{p(s)}}}w_j^{\frac{\frac{1}{q(s)}}{\frac{1}{p(s)}}}.
\]

Unwinding the definitions, it follows from \eqref{eq:joeq} and property \ref{it:ru1} of our multilinear Rubio de Francia algorithm that
\begin{equation}\label{eq:exduit}
\begin{split}
\|f_j\|_{L^{q_j}(W_j)}&=\|g_j^{\frac{\frac{1}{p_j}}{\frac{1}{q_j}}}(R_{p_j(s),r_j(s),\vec{w}}(g_j))^{-\frac{\frac{1}{p_j}-\frac{1}{q_j}}{\frac{1}{q_j}}}\|^{\frac{\frac{1}{q_j}}{\frac{1}{p_j(s)}}}_{L^{p_j(s)}(w_j^{p_j(s)})}\\
&\leq\|g_j\|^{\frac{\frac{1}{q_j}}{\frac{1}{p_j(s)}}}_{L^{p_j(s)}(w_j^{p_j(s)})}=\|f_j\|^{\frac{\frac{1}{q_j}}{\frac{1}{p_j}}}_{L^{p_j}(w_j^{p_j})}.
\end{split}
\end{equation}

Next, it follows from \eqref{eq:joeq}, H\"older's inequality, and property \ref{it:ru2} that
\begin{align*}
\|f_{m+1}\|_{L^{q'}(W^{-q'})}&\leq\|f_{m+1}w^{-1}\|_{L^{p'}}\|W^{-1}w\|_{L^{\frac{1}{\frac{1}{p}-\frac{1}{q}}}}\\
&=\|f_{m+1}\|_{L^{p'}(w^{-p'})}\Big\|\Big(\prod_{j=1}^mR_{p_j(s),r_j(s),\vec{w}}(g_j)\Big)^{\frac{\frac{1}{p(s)}-\frac{1}{q(s)}}{\frac{1}{p(s)}}}w^{\frac{\frac{1}{p(s)}-\frac{1}{q(s)}}{\frac{1}{p(s)}}}\Big\|_{L^{\frac{1}{\frac{1}{p(s)}-\frac{1}{q(s)}}}}\\
&=\|f_{m+1}\|_{L^{p'}(w^{-p'})}\Big\|\prod_{j=1}^mR_{p_j(s),r_j(s),\vec{w}}(g_j)\Big\|^{\frac{\frac{1}{p(s)}-\frac{1}{q(s)}}{\frac{1}{p(s)}}}_{L^{p(s)}(w^{p(s)})}\\
&\leq\|f_{m+1}\|_{L^{p'}(w^{-p'})}\prod_{j=1}^m\|R_{p_j(s),r_j(s),\vec{w}}(g_j)\|^{\frac{\frac{1}{p(s)}-\frac{1}{q(s)}}{\frac{1}{p(s)}}}_{L^{p_j(s)}(w_j^{p_j(s)})}\\
&\leq 2^m\|f_{m+1}\|_{L^{p'}(w^{-p'})}\prod_{j=1}^m\|f_j\|^{\frac{\frac{1}{p_j}-\frac{1}{q_j}}{\frac{1}{p_j}}}_{L^{p_j}(w_j^{p_j})}.
\end{align*}
By combining this estimate with \eqref{eq:exduit}, we have proven \eqref{eq:bigwest}.

Finally, we prove \eqref{eq:ru4}. Noting that
\[
\frac{1}{r_j}-\frac{1}{q_j}=\frac{\frac{1}{p(s)}-\frac{1}{q(s)}}{\frac{1}{p(s)}}\frac{1}{r_j(s)}+\frac{\frac{1}{q(s)}}{\frac{1}{p(s)}}\left(\frac{1}{r_j}-\frac{1}{p_j}\right),
\]
it follows from H\"older's inequality and \ref{it:ru3} that for a cube $Q$ we have
\begin{equation}\label{eq:rub1}
\begin{split}
\prod_{j=1}^m\langle &W_j^{-1}\rangle_{\frac{1}{\frac{1}{r_j}-\frac{1}{q_j}},Q} \leq\prod_{j=1}^m\langle R_{p_j,r_j,\vec{w}}(g_j)\rangle^{\frac{\frac{1}{p(s)}-\frac{1}{q(s)}}{\frac{1}{p(s)}}}_{r_j(s),Q}\langle w_j^{-1}\rangle^{\frac{\frac{1}{q(s)}}{\frac{1}{p(s)}}}_{\frac{1}{\frac{1}{r_j}-\frac{1}{p_j}},Q}\\
&\lesssim \Big(c_{\vec{p}(s),\vec{r}(s)}[\vec{w}]^{\max_{j=1,\ldots,m}\left\{\frac{\frac{1}{r_j(s)}}{\frac{1}{r_j(s)}-\frac{1}{p_j(s)}}\right\}}_{\vec{p}(s),(\vec{r}(s),\infty)} \inf_{y\in Q}\prod_{j=1}^mR_{p_j(s),r_j(s),\vec{w}}(g_j)(y)\Big)^{\frac{\frac{1}{p(s)}-\frac{1}{q(s)}}{\frac{1}{p(s)}}}\prod_{j=1}^m\langle w_j^{-1}\rangle^{\frac{\frac{1}{q(s)}}{\frac{1}{p(s)}}}_{\frac{1}{\frac{1}{r_j}-\frac{1}{p_j}},Q}.
\end{split}
\end{equation}
Moreover, we have
\[
\Big(\inf_{y\in Q}\prod_{j=1}^mR_{p_j(s),r_j(s),\vec{w}}(g_j)(y)\Big)^{\frac{\frac{1}{p(s)}-\frac{1}{q(s)}}{\frac{1}{p(s)}}}\langle W\rangle_{\frac{1}{\frac{1}{q}-\frac{1}{s}},Q}\leq\langle w^{\frac{\frac{1}{q(s)}}{\frac{1}{p(s)}}}\rangle_{q(s),Q}=\langle w\rangle^{\frac{\frac{1}{q(s)}}{\frac{1}{p(s)}}}_{\frac{1}{\frac{1}{p}-\frac{1}{s}},Q}.
\]
By combining this with \eqref{eq:rub1} we find that
\begin{equation}\label{eq:finalbigwest}
\Big(\prod_{j=1}^m\langle W_j^{-1}\rangle_{\frac{1}{\frac{1}{r_j}-\frac{1}{q_j}},Q}\Big)\langle W\rangle_{\frac{1}{\frac{1}{q}-\frac{1}{s}},Q}\lesssim\Big(c_{\vec{p}(s),\vec{r}(s)}[\vec{w}]^{\max_{j=1,\ldots,m}\left\{\frac{\frac{1}{r_j(s)}}{\frac{1}{r_j(s)}-\frac{1}{p_j(s)}}\right\}}_{\vec{p}(s),(\vec{r}(s),\infty)} \Big)^{\frac{\frac{1}{p(s)}-\frac{1}{q(s)}}{\frac{1}{p(s)}}}[\vec{w}]^{\frac{\frac{1}{q(s)}}{\frac{1}{p(s)}}}_{\vec{p},(\vec{r},s)}.
\end{equation}
By the translation lemma, Lemma \ref{lem:translation}, we have $[\vec{w}]_{\vec{p}(s),(\vec{r}(s),\infty)}=[\vec{w}]_{\vec{p},(\vec{r},s)}$ and, moreover, by using \eqref{eq:joeq} we compute
\begin{align*}
\frac{\frac{1}{r_j(s)}}{\frac{1}{r_j(s)}-\frac{1}{p_j(s)}}\frac{\frac{1}{p(s)}-\frac{1}{q(s)}}{\frac{1}{p(s)}}+\frac{\frac{1}{q(s)}}{\frac{1}{p(s)}}&=\frac{\left(\frac{1}{p_j(s)}-\frac{1}{q_j(s)}\right)\frac{1}{r_j(s)}+\left(\frac{1}{r_j(s)}-\frac{1}{p_j(s)}\right)\frac{1}{q_j(s)}}{\left(\frac{1}{r_j}-\frac{1}{p_j}\right)\frac{1}{p_j(s)}}\\
&=\frac{\frac{1}{r_j}-\frac{1}{q_j}}{\frac{1}{r_j}-\frac{1}{p_j}},
\end{align*}
which we interpret as being equal to $1$ when $\frac{1}{q_j}=\frac{1}{p_j}=\frac{1}{r_j}$, so that
\[
\max_{j=1,\ldots,m}\left\{\frac{\frac{1}{r_j(s)}}{\frac{1}{r_j(s)}-\frac{1}{p_j(s)}}\right\}\frac{\frac{1}{p(s)}-\frac{1}{q(s)}}{\frac{1}{p(s)}}+\frac{\frac{1}{q(s)}}{\frac{1}{p(s)}}=\max_{j=1,\ldots,m} \left\{\frac{\frac{1}{r_j}-\frac{1}{q_j}}{\frac{1}{r_j}-\frac{1}{p_j}}\right\}.
\]
Hence, \eqref{eq:ru4} follows by taking a supremum over all cubes $Q$ in \eqref{eq:finalbigwest}. This concludes Step 1.

\emph{Step 2.} Now suppose $\pmb{q}$ is arbitrary. For each $j$ we either have $\frac{1}{p_j}<\frac{1}{q_j}$ or $\frac{1}{p_j}\geq\frac{1}{q_j}$. Assume without loss of generality that there is a $j_1\in\{1,\ldots,m\}$ such that
\begin{equation}\label{eq:initpqst3}
\frac{1}{p_j}\geq\frac{1}{q_j}\quad\text{if $j\in\{1,\ldots,j_1\}$},\qquad \frac{1}{p_j}<\frac{1}{q_j}\quad\text{if $j\in\{j_1+1,\ldots,m+1\}$}.
\end{equation}
Indeed, if this is not the case then, just as in Step 1, we may permute the indices to reduce back to this case.

The strategy will be to construct the $m+1$ weights $\pmb{W}$ in $m-j_1+1$ steps through repeated application of Step 1.

We define
\[
\theta_k:=\left\{\begin{array}{ll}\displaystyle\frac{\sum_{j=m-k+2}^{m+1}\frac{1}{q_j}-\frac{1}{p_j}}{\sum_{j=j_1+1}^{m+1}\frac{1}{q_j}-\frac{1}{p_j}}&\text{if $k\in\{1,\ldots,m-j_1+1\}$};\\
0&\text{if $k=0$},
\end{array}\right.
\]
so that $0=\theta_0\leq\theta_1\leq\cdots\leq\theta_{m-j_1+1}=1$. Thus, defining,
\[
\frac{1}{q_j^k}:=\frac{1}{q_j}+\theta_k\left(\frac{1}{p_j}-\frac{1}{q_j}\right),
\]
we have
\[
\frac{1}{q_j}=\frac{1}{q_j^0}\leq\frac{1}{q_j^1}\leq\cdots\leq\frac{1}{q^{m-j_1}_j}\leq\frac{1}{q^{m-j_1+1}_j}=\frac{1}{p_j}.
\]

Now, we define
\begin{align*}
\pmb{q}^1&:=(q^1_1,\ldots,q^1_{j_1},q_{j_1+1},\ldots,q_m,p_{m+1})\\
\pmb{q}^2&:=(q^2_1,\ldots,q^2_{j_1},q_{j_1+1},\ldots,q_{m-1},p_m,p_{m+1})\\
&\,\,\,\,\,\vdots\\
\pmb{q}^{m-j_1}&:=(q^{m-j_1}_1,\ldots,q^{m-j_1}_{j_1},q_{j_1+1},p_{j_1+2},\ldots,p_{m+1}).
\end{align*}
First we will check that the reciprocals of the coordinates of these $m+1$-tuples sum to $1$. Indeed, using $\sum_{j=1}^{m+1}\frac{1}{p_j}=\sum_{j=1}^{m+1}\frac{1}{q_j}=1$, we have
\begin{align*}
\sum_{j=1}^{j_1}\frac{1}{q^k_j}&=\sum_{j=1}^{j_1}\frac{1}{q_j}+\theta_k\sum_{j=1}^{j_1}\frac{1}{p_j}-\frac{1}{q_j}=\sum_{j=1}^{j_1}\frac{1}{q_j}+\theta_k\left(1-\sum_{j=j_1+1}^{m+1}\frac{1}{p_j}\right)-\theta_k\left(1-\sum_{j=j_1+1}^{m+1}\frac{1}{q_j}\right)\\
&=\sum_{j=1}^{j_1}\frac{1}{q_j}+\sum_{j=m-k+2}^{m+1}\frac{1}{q_j}-\frac{1}{p_j}=1-\sum_{j=j_1+1}^{m-k+1}\frac{1}{q_j}-\sum_{j=m-k+2}^{m+1}\frac{1}{p_j}
\end{align*}
so that
\[
\sum_{j=1}^{j_1}\frac{1}{q^k_j}+\sum_{j=j_1+1}^{m-k+1}\frac{1}{q_j}+\sum_{j=m-k+2}^{m+1}\frac{1}{p_j}=1,
\]
as desired.

Now, for $k\in\{1,\ldots,m-j_1+1\}$ we define
\[
\gamma_k:=\max_{j=1,\ldots,j_1}\frac{\frac{1}{r_j}-\frac{1}{q^{k-1}_j}}{\frac{1}{r_j}-\frac{1}{q^k_j}},
\]
which should be interpreted as being equal to $1$ when $\frac{1}{q^k_j}=\frac{1}{r_j}$, and we write $\vec{q}^k=(q_1^k,\ldots,q_m^k)$ for the $m$-tuple given by the first $m$ coordinates of $\pmb{q}^k$, with $\frac{1}{q^k}:=\sum_{j=1}^m\frac{1}{q_j^k}$.

We may apply Step 1 with $j_0=j_1+1$ to obtain weights $\pmb{W}^{m-j_1}=(W^{m-j_1}_1,\ldots,W^{m-j_1}_{m+1})$ such that
\begin{equation}\label{eq:iteration11}
\prod_{j=1}^{m+1}\|f_j\|_{L^{q^{m-j_1}_j}((W^{m-j_1}_j)^{q^{m-j_1}_j})}\leq 2^m\prod_{j=1}^{m+1}\|f_j\|_{L^{p_j}(w_j^{p_j})}
\end{equation}
and
\begin{equation}\label{eq:iteration12}
[\pmb{W}^{m-j_1}]_{\pmb{q}^{m-j_1},(\pmb{r},\infty)}\leq C_{\pmb{p},\pmb{q},\pmb{r}}[\pmb{w}]^{\gamma_{m-j_1+1}}_{\pmb{p},(\pmb{r},\infty)}.
\end{equation}

Next we apply Step 1 with $j_0=j_1+2$ to obtain weights $\pmb{W}^{m-j_1-1}$ with
\[
\prod_{j=1}^{m+1}\|f_j\|_{L^{q^{m-j_1-1}_j}((W^{m-j_1-1}_j)^{q^{m-j_1-1}_j})}\leq 2^m\prod_{j=1}^{m+1}\|f_j\|_{L^{q^{m-j_1}_j}((W^{m-j_1}_j)^{q^{m-j_1}_j})}
\]
and
\[
[\pmb{W}^{m-j_1-1}]_{\pmb{q}^{m-j_1-1},(\pmb{r},\infty)}\leq C_{\pmb{p},\pmb{q},\pmb{r}}[\pmb{W}^{m-j_1}]^{\gamma_{m-j_1}}_{\pmb{q}^{m-j_1},(\pmb{r},\infty)}.
\]
Combining these estimates with \eqref{eq:iteration11} and \eqref{eq:iteration12} we obtain
\[
\prod_{j=1}^{m+1}\|f_j\|_{L^{q^{m-j_1-1}_j}((W^{m-j_1-1}_j)^{q^{m-j_1-1}_j})}\leq (2^m)^2\prod_{j=1}^{m+1}\|f_j\|_{L^{p_j}(w_j^{p_j})}
\]
and
\[
[\pmb{W}^{m-j_1-1}]_{\pmb{q}^{m-j_1-1},(\pmb{r},\infty)}\leq C_{\pmb{p},\pmb{q},\pmb{r}}[\pmb{w}]^{\gamma_{m-j_1}\gamma_{m-j_1+1}}_{\pmb{p},(\pmb{r},\infty)}.
\]

Continuing this process, applying Step 1 with $j_0=j_1+k$ for $k=3,\ldots,m-j_1+1$, we conclude, setting $\pmb{W}:=\pmb{W}^0$, that
\begin{equation}\label{eq:fin2mrelab1}
\prod_{j=1}^{m+1}\|f_j\|_{L^{q_j}(W_j^{q_j})}=\prod_{j=1}^{m+1}\|f_j\|_{L^{q^0_j}((W^0_j)^{q^0_j})}\leq (2^m)^{m-j_1+1}\prod_{j=1}^{m+1}\|f_j\|_{L^{p_j}(w_j^{p_j})}
\end{equation}
and
\begin{equation}\label{eq:fin2mrelab2}
[\pmb{W}]_{\pmb{q},(\pmb{r},\infty)}=[\pmb{W}^0]_{\pmb{q}^0,(\pmb{r},\infty)}\leq C_{\pmb{p},\pmb{q},\pmb{r}}[\pmb{w}]^{\prod_{k=1}^{m-j_1+1}\gamma_k}_{\pmb{p},(\pmb{r},\infty)}.
\end{equation}
Since $(2^m)^{m-j_1+1}\leq 2^{m^2}$, we note that \eqref{eq:mainprop1} now follows from \eqref{eq:fin2mrelab1}. Finally, we note that \eqref{eq:mainprop2} follows from \eqref{eq:fin2mrelab2}, provided we can show that
\begin{equation}\label{eq:claimst3}
\prod_{k=1}^{m-j_1+1}\gamma_k=\max_{j=1,\ldots,m+1}\frac{\frac{1}{r_j}-\frac{1}{q_j}}{\frac{1}{r_j}-\frac{1}{p_j}}.
\end{equation}
Note that by our initial assumption \eqref{eq:initpqst3}, this maximum is attained at some $j_2\in\{1,\ldots,j_1\}$.

We claim that
\[
\gamma_k=\frac{\frac{1}{r_{j_2}}-\frac{1}{q^{k-1}_{j_2}}}{\frac{1}{r_{j_2}}-\frac{1}{q^k_{j_2}}}
\]
for all $k\in\{1,\ldots,m-j_1+1\}$. Assuming for the moment that the claim is true, we find that
\[
\prod_{k=1}^{m-j_1+1}\gamma_k=\prod_{k=1}^{m-j_1+1}\frac{\frac{1}{r_{j_2}}-\frac{1}{q^{k-1}_{j_2}}}{\frac{1}{r_{j_2}}-\frac{1}{q^k_{j_2}}}=\frac{\frac{1}{r_{j_2}}-\frac{1}{q^{0}_{j_2}}}{\frac{1}{r_{j_2}}-\frac{1}{q^{m-j_1+1}_{j_2}}}= \frac{\frac{1}{r_{j_2}}-\frac{1}{q_{j_2}}}{\frac{1}{r_{j_2}}-\frac{1}{p_{j_2}}},
\]
proving \eqref{eq:claimst3}.

To prove the claim, we compute
\begin{align*}
\frac{1}{r_j}-\frac{1}{q_j^k}&=\frac{1}{r_j}-\frac{1}{q_j}-\theta_k\left(\frac{1}{r_j}-\frac{1}{q_j}\right)+\theta_k\left(\frac{1}{r_j}-\frac{1}{p_j}\right)\\
&=\left(\frac{1}{r_j}-\frac{1}{p_j}\right)\left((1-\theta_k)\frac{\frac{1}{r_j}-\frac{1}{q_j}}{\frac{1}{r_j}-\frac{1}{p_j}}+\theta_k\right)
\end{align*}
so that
\[
\frac{\frac{1}{r_j}-\frac{1}{q^{k-1}_j}}{\frac{1}{r_j}-\frac{1}{q^k_j}}=\frac{(1-\theta_{k-1})\frac{\frac{1}{r_j}-\frac{1}{q_j}}{\frac{1}{r_j}-\frac{1}{p_j}}+\theta_{k-1}}{(1-\theta_k)\frac{\frac{1}{r_j}-\frac{1}{q_j}}{\frac{1}{r_j}-\frac{1}{p_j}}+\theta_k} =f_k\left(\frac{\frac{1}{r_j}-\frac{1}{q_j}}{\frac{1}{r_j}-\frac{1}{p_j}}\right),
\]
where
\[
f_k(x)=\frac{(1-\theta_{k-1})x+\theta_{k-1}}{(1-\theta_k)x+\theta_k}.
\]
We note that proving the claim is equivalent to proving the equality
\[
\max_{j=1,\ldots,m+1}f_k\left(\frac{\frac{1}{r_j}-\frac{1}{q_j}}{\frac{1}{r_j}-\frac{1}{p_j}}\right)=f_k\left(\max_{j=1,\ldots,m+1}\frac{\frac{1}{r_j}-\frac{1}{q_j}}{\frac{1}{r_j}-\frac{1}{p_j}}\right).
\]
The inequality
\[
f_k\left(\max_{j=1,\ldots,m+1}\frac{\frac{1}{r_j}-\frac{1}{q_j}}{\frac{1}{r_j}-\frac{1}{p_j}}\right)=f_k\left(\frac{\frac{1}{r_{j_2}}-\frac{1}{q_{j_2}}}{\frac{1}{r_{j_2}}-\frac{1}{p_{j_2}}}\right)\leq\max_{j=1,\ldots,m+1}f_k\left(\frac{\frac{1}{r_j}-\frac{1}{q_j}}{\frac{1}{r_j}-\frac{1}{p_j}}\right)
\]
is clear. To prove the converse inequality, it suffices to show that $f_k$ is an increasing function for all $k\in\{1,\ldots,m-j_1+1\}$. Computing
\begin{align*}
f_k'(x)&=\frac{(1-\theta_{k-1})((1-\theta_k)x+\theta_k)-(1-\theta_k)((1-\theta_{k-1})x+\theta_{k-1})}{((1-\theta_k)x+\theta_k)^2}\\
&=\frac{\theta_k-\theta_{k-1}}{((1-\theta_k)x+\theta_k)^2}\geq 0,
\end{align*}
we have proven the desired result. This concludes Step 2. The assertion follows.
\end{proof}

\begin{proof}[Proof of Theorem \ref{thm:qmlre}]
The result essentially follows from an application of Theorem \ref{thm:main}. However, in order to use this result we must reduce to a case where $\frac{1}{p}\leq1$ so that we may set $\frac{1}{p_{m+1}}=1-\frac{1}{p}\geq 0$. To reduce to this case, we employ a general rescaling trick that also appears in the proof of the case $m=1$ given by Auscher and Martell in \cite[Theorem 4.9]{AM07}.

First we will show that we may assume that $\frac{1}{r}:=\sum_{j=1}^m\frac{1}{r_j}=1$. Indeed, assuming we have shown the result for $\frac{1}{r}=1$, we consider the $m+1$-tuple $(|f_1|^r,\ldots,|f_m|^r,|h|^r)$. Then, since
\[
[\vec{w}]^{\frac{1}{r}}_{\frac{\vec{q}}{r},(\frac{\vec{r}}{r},\frac{s}{r})}=[(w_1^{\frac{1}{r}},\ldots,w_m^{\frac{1}{r}})]_{\vec{q},(\vec{r},s)},
\]
we find that for all $\vec{w}\in A_{\frac{\vec{q}}{r},(\frac{\vec{r}}{r},\frac{s}{r})}$ we have
\begin{align*}
\||h|^r\|_{L^{\frac{q}{r}}(w^{\frac{q}{r}})}&=\|h\|^r_{L^q((w^{\frac{1}{r}})^q)}\leq\phi_{\vec{q}}([(w_1^{\frac{1}{r}},\ldots,w_m^{\frac{1}{r}})]_{\vec{q},(\vec{r},s)})^r\prod_{j=1}^m\|f_j\|^r_{L^{q_j}((w_j^{\frac{1}{r}})^{q_j})}\\
&=\phi_{\vec{q}}([\vec{w}]^{\frac{1}{r}}_{\frac{\vec{q}}{r},(\frac{\vec{r}}{r},\frac{s}{r})})^r\prod_{j=1}^m\||f_j|^r\|_{L^{\frac{q_j}{r}}(w_j^{\frac{q_j}{r}})}.
\end{align*}
Thus, since $\sum_{j=1}^m \frac{r}{r_j}=1$, applying the extrapolation result with $\vec{r}$ replaced by $\frac{\vec{r}}{r}$, $\vec{q}$ replaced by $\frac{\vec{q}}{r}$, and $s$ replaced by $\frac{s}{r}$, we find that for any $\frac{\vec{p}}{r}$ with $\frac{\vec{p}}{r}>\frac{\vec{r}}{r}$ and $\frac{p}{r}<\frac{s}{r}$, or equivalently, for all $\vec{p}>(\vec{r},s)$, we have
\begin{align*}
\|h\|_{L^p(w^p)}&=\||h|^r\|^{\frac{1}{r}}_{L^{\frac{p}{r}}((w^r)^{\frac{p}{r}})}\leq\phi_{\frac{\vec{p}}{r},\frac{\vec{q}}{r},\frac{\vec{r}}{r},\frac{s}{r}}([(w_1^r,\ldots,w_m^r)]_{\frac{\vec{p}}{r},(\frac{\vec{r}}{r},\frac{s}{r})})^{\frac{1}{r}}\prod_{j=1}^m\||f_j|^r\|^{\frac{1}{r}}_{L^{\frac{p_j}{r}}((w_j^r)^{\frac{p_j}{r}})}\\
&=\phi_{\frac{\vec{p}}{r},\frac{\vec{q}}{r},\frac{\vec{r}}{r},\frac{s}{r}}([\vec{w}]^r_{\vec{p},(\vec{r},s)})^{\frac{1}{r}}\prod_{j=1}^m\|f_j\|_{L^{p_j}(w_j^{p_j})},
\end{align*}
for all $\vec{w}\in A_{\vec{p},(\vec{r},s)}$, with
\[
\phi_{\frac{\vec{p}}{r},\frac{\vec{q}}{r},\frac{\vec{r}}{r},\frac{s}{r}}([\vec{w}]^r_{\vec{p},(\vec{r},s)})^{\frac{1}{r}}=2^{\frac{m^2}{r}}\phi_{\vec{q}}\Big(C_{\vec{p},\vec{q},\vec{r},s}[\vec{w}]^{ r\max\left(\frac{\frac{1}{r_1}-\frac{1}{q_1}}{\frac{1}{r_1}-\frac{1}{p_1}},\ldots,\frac{\frac{1}{r_m}-\frac{1}{q_m}}{\frac{1}{r_m}-\frac{1}{p_m}},\frac{\frac{1}{q}-\frac{1}{s}}{\frac{1}{p}-\frac{1}{s}}\right)}_{\vec{p},(\vec{r},s)}\Big)^{\frac{1}{r}}
\]
as desired.

Now that we have reduced to the case where $\frac{1}{r}=1$, we have $\frac{1}{s}\leq\frac{1}{p}\leq \sum_{j=1}^m\frac{1}{r_j}=1$. Thus, we may set $\frac{1}{p_{m+1}}:=1-\frac{1}{p}\geq 0$, $\frac{1}{q_{m+1}}:=1-\frac{1}{q}\geq0$, $\frac{1}{r_{m+1}}:=1-\frac{1}{s}\geq0$ and $w_{m+1}:=w^{-1}$.

Let $f_{m+1}\in L^{p_{m+1}}(w_{m+1}^{p_{m+1}})$ and let $\pmb{W}=(W_1,\ldots,W_{m+1})$ be the weights obtained from Theorem~\ref{thm:main}. Setting $\vec{W}=(W_1,\ldots,W_m)$ and $W:=\prod_{j=1}^mW_j$ we find, using the assumption \eqref{eq:multextrapend} and property \eqref{eq:mainprop1} of $\pmb{W}$, that
\begin{equation}\label{eq:mainthmfin}
\begin{split}
|\langle h,f_{m+1}\rangle|&\leq\|h\|_{L^q(W^q)}\|f_{m+1}\|_{L^{q_{m+1}}(W_{m+1}^{q_{m+1}})}\leq\phi_{\vec{q}}([\vec{W}]_{\vec{q},(\vec{r},s)})\prod_{j=1}^{m+1}\|f_j\|_{L^{q_j}(W_j^{q_j})}\\
&\leq 2^{m^2}\phi_{\vec{q}}([\vec{W}]_{\vec{q},(\vec{r},s)})\prod_{j=1}^{m+1}\|f_j\|_{L^{p_j}(w_j^{p_j})}.
\end{split}
\end{equation}
Moreover, it follows from \eqref{eq:mainprop2} that
\begin{align*}
[\vec{W}]_{\vec{q},(\vec{r},s)}&=[\pmb{W}]_{\pmb{q},(\pmb{r},\infty)}\leq C_{\pmb{p},\pmb{q},\pmb{r}}[\pmb{w}]^{\max_{j=1,\ldots,m+1} \left\{\frac{\frac{1}{r_j}-\frac{1}{q_j}}{\frac{1}{r_j}-\frac{1}{p_j}}\right\}}_{\pmb{p},(\pmb{r},\infty)}\\
&=C_{\vec{p},\vec{q},\vec{r},s}[\vec{w}]^{\max\left(\frac{\frac{1}{r_1}-\frac{1}{q_1}}{\frac{1}{r_1}-\frac{1}{p_1}},\ldots,\frac{\frac{1}{r_m}-\frac{1}{q_m}}{\frac{1}{r_m}-\frac{1}{p_m}},\frac{\frac{1}{q}-\frac{1}{s}}{\frac{1}{p}-\frac{1}{s}}\right)}_{\vec{p},(\vec{r},\infty)}.
\end{align*}
By combining this estimate with \eqref{eq:mainthmfin} and by noting that
\[
\|h\|_{L^{p}(w^{p})}=\sup_{\|f_{m+1}\|_{L^{p_{m+1}}(w_{m+1}^{p_{m+1}})}=1}|\langle h,f_{m+1}\rangle|,
\]
the assertion follows.
\end{proof}
\section{Applications of the extrapolation theorem}\label{sec:app}
In applying extrapolation theorems, one can obtain further results by making appropriate choices in the $m+1$-tuples. We provide some applications in this section.
\subsection{Boundedness of operators through extrapolation}
Given an operator $T$ defined on $m$-tuples of functions, one can apply the extrapolation result to the $m+1$-tuples $(f_1,\ldots,f_m,T(f_1,\ldots,f_m))$ to obtain the following extension result:
\begin{theorem}\label{thm:mainopthm}
Let $T$ be an $m$-linear or a positive valued $m$-sublinear operator and suppose that there exist $r_1,\ldots,r_m\in(0,\infty)$, $s\in(0,\infty]$ and $q_1,\ldots,q_m\in(0,\infty]$ with $\vec{q}\geq(\vec{r},s)$ and an increasing function $\phi_{\vec{q}}$ such that
\begin{equation}\label{eq:mainopthm}
\|T\|_{L^{q_1}(w_1^{q_1})\times\cdots\times L^{q_m}(w_m^{q_m})\to L^q(w^q)}\leq\phi_{\vec{q}}([\vec{w}]_{\vec{q},(\vec{r},s)})
\end{equation}
for all $\vec{w}\in A_{\vec{q},(\vec{r},s)}$.

Then for all $p_1,\ldots,p_m\in(0,\infty]$ with $\vec{p}>(\vec{r},s)$ and all weights $\vec{w}\in A_{\vec{p},(\vec{r},s)}$ the operator $T$ extends to a bounded operator $L^{p_1}(w_1^{p_1})\times\cdots\times L^{p_m}(w_m^{p_m})\to L^p(w^p)$. Moreover, $T$ satisfies the bound
\[
\|T\|_{L^{p_1}(w_1^{p_1})\times\cdots\times L^{p_m}(w_m^{p_m})\to L^p(w^p)}\leq2^{\frac{m^2}{r}}\phi_{\vec{q}}\Big(C_{\vec{p},\vec{q},\vec{r},s}[\vec{w}]_{\vec{p},(\vec{r},s)}^{ r\max\left(\frac{\frac{1}{r_1}-\frac{1}{q_1}}{\frac{1}{r_1}-\frac{1}{p_1}},\ldots,\frac{\frac{1}{r_m}-\frac{1}{q_m}}{\frac{1}{r_m}-\frac{1}{p_m}},\frac{\frac{1}{q}-\frac{1}{s}}{\frac{1}{p}-\frac{1}{s}}\right)}\Big)^{\frac{1}{r}},
\]
where $\frac{1}{r}=\sum_{j=1}^m\frac{1}{r_j}$.
\end{theorem}
\begin{proof}
Let $f_1,\ldots,f_m$ be simple functions. By \eqref{eq:mainopthm} we have
\[
\|T(f_1,\ldots,f_m)\|_{L^q(w^q)}\leq\phi_{\vec{q}}([\vec{w}]_{\vec{q},(\vec{r},s)})\prod_{j=1}^m\|f_j\|_{L^{q_j}(w_j^{q_j})}
\]
for all $\vec{w}\in A_{\vec{q},(\vec{r},s)}$. Thus, by applying Theorem \ref{thm:qmlre} to the $m+1$-tuple $(f_1,\ldots,f_m,T(f_1,\ldots,f_m))$ we find that for all $p_1,\ldots,p_m\in(0,\infty]$ with $\vec{p}>(\vec{r},s)$ and all weights $\vec{w}\in A_{\vec{p},(\vec{r},s)}$ we have
\[
\|T(f_1,\ldots,f_m)\|_{L^p(w^p)}\leq\phi_{\vec{p},\vec{q},\vec{r},s}([\vec{w}]_{\vec{p},(\vec{r},s)})\prod_{j=1}^m\|f_j\|_{L^{p_j}(w_j^{p_j})}
\]
with $\phi_{\vec{p},\vec{q},\vec{r},s}$ given by \eqref{eq:multextraquant}. Since this estimate holds for all simple functions $f_1,\ldots,f_m$, the assumptions on $T$ allow us to conclude the results through density.
\end{proof}
The initial estimate \eqref{eq:mainopthm} is often obtained through sparse domination. Once we have an estimate of the form
\[
|\langle T(f_1,\ldots,f_m),g\rangle|\lesssim\sup_{\Sp}\Lambda_{(r_1,\ldots,r_m,s'),\Sp}(f_1,\ldots,f_m,g),
\]
it follows from duality and Proposition \ref{prop:mainsym} that for $p_1,\ldots,p_m\in(0,\infty]$ with $\vec{p}>(\vec{r},s)$ and $\frac{1}{p}<1$, we have
\begin{equation}\label{eq:dualrangesparseest}
\|T\|_{L^{p_1}(w_1^{p_1})\times\cdots\times L^{p_m}(w_m^{p_m})\to L^p(w^p)}\lesssim[\vec{w}]^{\max\left(\frac{\frac{1}{r_1}}{\frac{1}{r_1}-\frac{1}{p_1}},\ldots,\frac{\frac{1}{r_m}}{\frac{1}{r_m}-\frac{1}{p_m}},\frac{\frac{1}{s'}}{\frac{1}{s'}-\frac{1}{p'}}\right)}_{\vec{p},(\vec{r},s)}.
\end{equation}
We are, however, still missing the cases outside of the reflexive range $\frac{1}{p}<1$. One can reach these cases through extrapolation, see \cite{LMO18, LMMOV19}. The novelty in our result is that we also obtain a quantitative weighted bound in this range through Theorem~\ref{thm:mainopthm}.
\begin{corollary}\label{cor:sparseext}
Let $T$ be an $m$-linear or a positive valued $m$-sublinear operator and suppose that there exist $r_1,\ldots,r_m\in(0,\infty)$, $s\in[1,\infty]$ such that for all bounded compactly supported $f_1,\ldots,f_m,g$ we have
\[
|\langle T(f_1,\ldots,f_m),g\rangle|\lesssim\sup_{\Sp}\Lambda_{(r_1,\ldots,r_m,s'),\Sp}(f_1,\ldots,f_m,g),
\]
where the supremum runs over all sparse collections $\Sp$ with a fixed sparsity constant. Then for all $p_1,\ldots,p_m\in(0,\infty]$ with $\vec{p}>(\vec{r},s)$ and all weights $\vec{w}\in A_{\vec{p},(\vec{r},s)}$ the operator $T$ extends to a bounded operator $L^{p_1}(w_1^{p_1})\times\cdots\times L^{p_m}(w_m^{p_m})\to L^p(w^p)$. Moreover, $T$ satisfies the bound
\begin{equation}\label{eq:corsparse}
\|T\|_{L^{p_1}(w_1^{p_1})\times\cdots\times L^{p_m}(w_m^{p_m})\to L^p(w^p)}\lesssim[\vec{w}]_{\vec{p},(\vec{r},s)}^{ \max\left(\frac{\frac{1}{r_1}}{\frac{1}{r_1}-\frac{1}{p_1}},\ldots,\frac{\frac{1}{r_m}}{\frac{1}{r_m}-\frac{1}{p_m}},\frac{1-\frac{1}{s}}{\frac{1}{p}-\frac{1}{s}}\right)}.
\end{equation}
\end{corollary}
\begin{proof}
We set $\frac{1}{\tau}:=\frac{1}{s'}+\sum_{j=1}^m\frac{1}{r_j}$. Assuming the set of $\vec{p}$ satisfying $\vec{p}>(\vec{r},s)$ is non-empty, we have $\tau<1$. Indeed, for such a $\vec{p}$ we have
\[
\frac{1}{\tau}>\frac{1}{p'}+\sum_{j=1}^m\frac{1}{p_j}=1.
\]
Setting $\frac{1}{q_j}:=\frac{\tau}{r_j}<\frac{1}{r_j}$, we have
\[
\frac{1}{q}=\frac{1}{\frac{1}{s'}+\sum_{j=1}^m\frac{1}{r_j}}\sum_{j=1}^m\frac{1}{r_j}=1-\frac{\tau}{s'}
\]
so that
\[
\frac{\frac{1}{r_1}}{\frac{1}{r_1}-\frac{1}{q_1}}=\cdots=\frac{\frac{1}{r_m}}{\frac{1}{r_m}-\frac{1}{q_m}}=\frac{\frac{1}{s'}}{\frac{1}{s'}-\frac{1}{q'}}=\frac{1}{1-\tau}.
\]
Then by \eqref{eq:dualrangesparseest} with this specific choice of $q_j$ we obtain
\[
\|T\|_{L^{q_1}(w_1^{q_1})\times\cdots\times L^{q_m}(w_m^{q_m})\to L^q(w^q)}\lesssim[\vec{w}]^{\frac{1}{1-\tau}}_{\vec{q},(\vec{r},s)}
\]
for all $\vec{w}\in A_{\vec{q},(\vec{r},s)}$. Thus, it follows from \eqref{thm:mainopthm} that for all $p_1,\ldots,p_m\in(0,\infty]$ with $\vec{p}>(\vec{r},s)$ and all weights $\vec{w}\in A_{\vec{p},(\vec{r},s)}$ we have
\begin{equation}\label{eq:corsp1}
\|T\|_{L^{p_1}(w_1^{p_1})\times\cdots\times L^{p_m}(w_m^{p_m})\to L^p(w^p)}\lesssim[\vec{w}]_{\vec{p},(\vec{r},s)}^{ \frac{1}{1-\tau}\max\left(\frac{\frac{1}{r_1}-\frac{1}{q_1}}{\frac{1}{r_1}-\frac{1}{p_1}},\ldots,\frac{\frac{1}{r_m}-\frac{1}{q_m}}{\frac{1}{r_m}-\frac{1}{p_m}},\frac{\frac{1}{q}-\frac{1}{s}}{\frac{1}{p}-\frac{1}{s}}\right)}.
\end{equation}
Noting that
\[
\frac{\frac{1}{r_j}-\frac{1}{q_j}}{\frac{1}{r_j}-\frac{1}{p_j}}=(1-\tau)\frac{\frac{1}{r_j}}{\frac{1}{r_j}-\frac{1}{p_j}},\qquad\frac{\frac{1}{q}-\frac{1}{s}}{\frac{1}{p}-\frac{1}{s}}=(1-\tau)\frac{1-\frac{1}{s}}{\frac{1}{p}-\frac{1}{s}},
\]
the estimate \eqref{eq:corsparse} now follows from \eqref{eq:corsp1}.
\end{proof}
\begin{remark}
We note that in particular the quantitative bound \eqref{eq:corsparse} extends the bound \eqref{eq:dualrangesparseest} obtained from the sparse form, even though the proof only used the sparse bound for the particular values $\frac{1}{p_j}=\frac{\tau}{r_j}$ with $\frac{1}{\tau}=\frac{1}{s'}+\sum_{j=1}^m\frac{1}{r_j}$. It seems that these values are, in some sense, central for the sparse form and the quantitative bound for these values has already appeared in \cite{LMO18}, but giving a quantitative bound for the whole range of $\vec{p}>(\vec{r},s)$ is new. In the case $m=1$ this value becomes $p=r(\frac{1}{s'}+\frac{1}{r})=1+\frac{r}{s'}$ which is the value central in the main theorem of \cite{frey16}. In particular when $r=1$, $s=\infty$ we have $p=2$ which is central in the theory of Calder\'on-Zygmund operators.
\end{remark}

In the full-range case, i.e., when $r_1=\ldots=r_m=1$, $s=\infty$, the particular case we consider becomes $p_1=\ldots=p_m=m+1$ and in \cite{DLP15} a bound in this case for multilinear Calder\'on-Zygmund operators was found. Using the sparse domination result of \cite{DLP15}, this result was extended by Li, Moen, and Sun in \cite{LMS14} to the range of $p_j\in(1,\infty)$ with $\frac{1}{p}\leq 1$. They showed that for a multilinear Calder\'on-Zygmund operator $T$, all $p_1,\ldots,p_m\in (1,\infty)$ with $\frac{1}{p}\leq 1$ and all weights $\vec{w}\in A_{\vec{p}}$ we have
\begin{equation}\label{eq:multiczowbd}
\|T\|_{L^{p_1}(w_1)\times\cdots\times L^{p_m}(w_m)\to L^p(v_{\vec{w}})}\lesssim[\vec{w}]_{A_{\vec{p}}}^{ \max\left(\frac{p_1'}{p},\ldots,\frac{p_m'}{p},1\right)},
\end{equation}
where the class $A_{\vec{p}}$ is defined through the constant in \eqref{eq:frmultiwdef}, and $v_{\vec{w}}:=\prod_{j=1}^m w_j^{\frac{p}{p_j}}$. They proved that this same bound holds even in the case $\frac{1}{p}>1$ for multilinear sparse operators, leading them to conjecture that the bound \eqref{eq:multiczowbd} should also extend to the case $\frac{1}{p}>1$. This conjecture was independently proven to be true by Conde-Alonso and Rey \cite{CR16} and Lerner and Nazarov \cite{LN15} for kernels satisfying $\log$--Dini conditions. We also refer the reader to \cite{lacey17, HRT17}, where the weaker Dini condition was considered in the linear case. The Dini condition was used in the multilinear setting by Dami\'an, Hormozi and Li \cite{DHL18} where, in addition, quantitative mixed multilinear $A_{\vec{p}}$--$A_\infty$ bounds were considered.

Our results yields another proof of the extension of the bound to the case $\frac{1}{p}>1$. To see this, we note that by replacing the $w_j$ by $w_j^{p_j}$ we have $v_{\vec{w}}=\prod_{j=1}^m(w_j^{p_j})^{\frac{p}{p_j}}=w^p$ and
\[
[(w_1^{p_1},\ldots,w_m^{p_m})]_{A_{\vec{p}}}=[\vec{w}]^p_{\vec{p},(\vec{1},\infty)}.
\]
Thus, the result \eqref{eq:multiczowbd} takes the equivalent form
\[
\|T\|_{L^{p_1}(w_1^{p_1})\times\cdots\times L^{p_m}(w_m^{p_m})\to L^p(w^p)}\lesssim[\vec{w}]_{\vec{p},(\vec{1},\infty)}^{ \max\left(p_1',\ldots,p_m',p\right)},
\]
which precisely corresponds to the bound \eqref{eq:dualrangesparseest}. By applying our extrapolation result we can now extend \eqref{eq:multiczowbd}, proving the following:
\begin{corollary}
Let $T$ be an $m$-linear Calder\'on-Zygmund operator. Then for all $p_1,\ldots,p_m\in(1,\infty)$ we have
\[
\|T\|_{L^{p_1}(w_1)\times\cdots\times L^{p_m}(w_m)\to L^p(v_{\vec{w}})}\lesssim[\vec{w}]_{A_{\vec{p}}}^{ \max\left(\frac{p_1'}{p},\ldots,\frac{p_m'}{p},1\right)}.
\]
\end{corollary}
As in Corollary~\ref{cor:sparseext}, our result actually yields weighted bounds for multilinear Calder\'on-Zygmund operators in the more general case $p_1,\ldots,p_m\in(1,\infty]$ with $\frac{1}{p}>0$.
\subsection{Vector-valued extrapolation}
By Fubini's Theorem we are able to extend the extrapolation theorem to a vector-valued setting. In the following result we are considering spaces of the form $L^p(w^p;L^q(\Omega))$ for $p,q\in(0,\infty]$, a weight $w$, and $\Omega$ a $\sigma$-finite measure space. Such spaces consist of functions $f:\R^n\to L^q(\Omega)$ such that the function $\|f\|_{L^q(\Omega)}$ lies in $L^p(w^p)$, with $\|f\|_{L^p(w^p;L^q(\Omega))}:=\big\|\|f\|_{L^q(\Omega)}\big\|_{L^p(w^p)}$. In the case when $p=q$, we can use Fubini's Theorem to find that
\[
\|f\|_{L^q(w^q;L^q(\Omega))}=\big\|\|f\|_{L^q(w^q)}\big\|_{L^q(\Omega)},
\]
valid for any $q\in(0,\infty]$, allowing us to carry over scalar-valued estimates to the vector-valued setting.
\begin{theorem}[Vector-valued extrapolation]\label{thm:vvextrap}
Let $r_1,\ldots,r_m\in(0,\infty)$, $s\in(0,\infty]$. Let $\Omega$ be a $\sigma$-finite measure space, let $q_1,\ldots,q_m\in(0,\infty]$ and $\vec{q}\geq(\vec{r},s)$, and let $(f_1,\ldots,f_m,h)$ be an $m+1$-tuple of measurable functions on $\R^n\times\Omega$. Assume that there is an increasing function $\phi_{\vec{q},\vec{r},s}$ such that the inequality
\begin{equation}\label{eq:multextrapinitvect}
\|h\|_{L^q(w^q)}\leq\phi_{\vec{q},\vec{r},s}([\vec{w}]_{\vec{q},(\vec{r},s)})\prod_{j=1}^m\|f_j\|_{L^{q_j}(w_j^{q_j})}
\end{equation}
holds pointwise a.e. in $\Omega$ for all $\vec{w}\in A_{\vec{q},(\vec{r},s)}$.

Then for all $p_1\,\ldots,p_m\in(0,\infty]$ with $\vec{p}>(\vec{r},s)$ there is an increasing function $\phi_{\vec{p},\vec{q},\vec{r},s}$ such that
\[
\|h\|_{L^p(w^p;L^q(\Omega))}\leq\phi_{\vec{p},\vec{q},\vec{r},s}([\vec{w}]_{\vec{p},(\vec{r},s)})\prod_{j=1}^m\|f_j\|_{L^{p_j}(w_j^{p_j};L^{q_j}(\Omega))}
\]
for all $\vec{w}\in A_{\vec{p},(\vec{r},s)}$. More explicitly, we can take
\[
\phi_{\vec{p},\vec{q},\vec{r},s}(t)=2^{\frac{m^2}{r}}\phi_{\vec{q},\vec{r},s}\Big(C_{\vec{p},\vec{q},\vec{r},s}t^{ r\max\left(\frac{\frac{1}{r_1}-\frac{1}{q_1}}{\frac{1}{r_1}-\frac{1}{p_1}},\ldots,\frac{\frac{1}{r_m}-\frac{1}{q_m}}{\frac{1}{r_m}-\frac{1}{p_m}},\frac{\frac{1}{q}-\frac{1}{s}}{\frac{1}{p}-\frac{1}{s}}\right)}\Big)^{\frac{1}{r}},
\]
where $\frac{1}{r}=\sum_{j=1}^m\frac{1}{r_j}$.
\end{theorem}
\begin{proof}
Set $\tilde{f}_j:=\|f_j\|_{L^{q_j}(\Omega)}$, $\tilde{h}:=\|h\|_{L^q(\Omega)}$, which, by Fubini's Theorem, are measurable functions on $\R^n$. Then by Fubini's Theorem, the assumption \eqref{eq:multextrapinitvect}, and H\"older's inequality, we have
\begin{align*}
\|\tilde{h}\|_{L^q(w^q)}&=\big\|\|h\|_{L^q(w^q)}\|\big\|_{L^q(\Omega)}\\
&\leq\phi_{\vec{q},\vec{r},s}([\vec{w}]_{\vec{q},(\vec{r},s)})\big\|\prod_{j=1}^m\|f_j\|_{L^{q_j}(w_j^{q_j})}\big\|_{L^q(\Omega)}\\
&\leq\phi_{\vec{q},\vec{r},s}([\vec{w}]_{\vec{q},(\vec{r},s)})\prod_{j=1}^m\big\|\|f_j\|_{L^{q_j}(w_j^{q_j})}\big\|_{L^{q_j}(\Omega)}\\
&=\phi_{\vec{q},\vec{r},s}([\vec{w}]_{\vec{q},(\vec{r},s)})\prod_{j=1}^m\|\tilde{f}_j\|_{L^{q_j}(w_j^{q_j})}.
\end{align*}
Thus, we may apply Theorem \ref{thm:qmlre} to the $m+1$-tuple $(\tilde{f}_1,\ldots,\tilde{f}_m,\tilde{h})$, proving the result.
\end{proof}
By iterated uses of Fubini's Theorem, a similar argument also allows us to extrapolate to vector-valued bounds with iterated $L^q$-spaces which were considered by Benea and Muscalu through their helicoidal method \cite{BM18}, but we do not detail this here.

We emphasize here that our extrapolation result goes through even if we have $q_j=\infty$ for some $j\in\{1,\ldots,m\}$ in \eqref{eq:multextrapinitvect}. The conclusion of our result then yields vector-valued estimates in the mixed normed spaces $L^{p_j}(L^\infty)$.

If we take $\Omega=\N$ with the counting measure, we obtain vector-valued bounds for $\ell^q$-spaces. Given an $m$-linear operator $T$ and sequences of measurable functions $(f^1_k)_{k\in\N},\ldots,(f^m_k)_{k\in\N}$, we may define
\begin{equation}\label{eq:textlq}
T((f^1_k)_{k\in\N},\ldots,(f^m_k)_{k\in\N}):=(T(f^1_k,\ldots,f^m_k))_{k\in\N}.
\end{equation}
By combining the vector-valued extrapolation theorem with Corollary~\ref{cor:sparseext}, we obtain the following:
\begin{corollary}\label{cor:vecext}
Let $T$ be an $m$-linear operator and suppose that there exist $r_1,\ldots,r_m\in(0,\infty)$, $s\in[1,\infty]$ such that for all bounded compactly supported $f_1,\ldots,f_m,g$ we have
\[
|\langle T(f_1,\ldots,f_m),g\rangle|\lesssim\sup_{\Sp}\Lambda_{(r_1,\ldots,r_m,s'),\Sp}(f_1,\ldots,f_m,g),
\]
where the supremum runs over all sparse collections $\Sp$ with a fixed sparsity constant.

Then for all $p_1\,\ldots,p_m,q_1,\ldots,q_m\in(0,\infty]$ with $\vec{p},\vec{q}>(\vec{r},s)$, the operator $T$ has a bounded extension $L^{p_1}(w_1^{p_1};\ell^{q_1})\times\cdots\times L^{p_m}(w_m^{p_m};\ell^{q_m})\to L^p(w^p;\ell^q)$ given by \eqref{eq:textlq}. Moreover, there is an increasing function $\phi_{\vec{p},\vec{q},\vec{r},s}$ such that
\[
\|T\|_{L^{p_1}(w_1^{p_1};\ell^{q_1})\times\cdots\times L^{p_m}(w_m^{p_m};\ell^{q_m})\to L^p(w^p;\ell^q)}\leq\phi_{\vec{p},\vec{q},\vec{r},s}([\vec{w}]_{\vec{p},(\vec{r},s)})
\]
for all $\vec{w}\in A_{\vec{p},(\vec{r},s)}$. More explicitly, we can take
\begin{equation}\label{eq:extratwice}
\phi_{\vec{p},\vec{q},\vec{r},s}(t)\eqsim t^{\max\left(\frac{\frac{1}{r_1}}{\frac{1}{r_1}-\frac{1}{q_1}},\ldots,\frac{\frac{1}{r_m}}{\frac{1}{r_m}-\frac{1}{q_m}},\frac{1-\frac{1}{s}}{\frac{1}{q}-\frac{1}{s}}\right)\cdot \max\left(\frac{\frac{1}{r_1}-\frac{1}{q_1}}{\frac{1}{r_1}-\frac{1}{p_1}},\ldots,\frac{\frac{1}{r_m}-\frac{1}{q_m}}{\frac{1}{r_m}-\frac{1}{p_m}},\frac{\frac{1}{q}-\frac{1}{s}}{\frac{1}{p}-\frac{1}{s}}\right)}.
\end{equation}
\end{corollary}
\begin{proof}
For each $j\in\{1,\ldots,m\}$, let $(f^j_k)_{k\in\N}$ be a sequence of simple functions with at most finitely many non-zero entries. Setting $f_j(x,k):=f_k^j(x)$ and $h(x,k):=T(f_k^1,\ldots,f_k^m)(x)$, it follows from Corollary~\ref{cor:sparseext} that \eqref{eq:multextrapinitvect} is satisfied with
\[
\phi_{\vec{q},\vec{r},s}(t)\eqsim t^{\max\left(\frac{\frac{1}{r_1}}{\frac{1}{r_1}-\frac{1}{q_1}},\ldots,\frac{\frac{1}{r_m}}{\frac{1}{r_m}-\frac{1}{q_m}},\frac{1-\frac{1}{s}}{\frac{1}{q}-\frac{1}{s}}\right)}.
\]
The assertion now follows from Theorem~\ref{thm:vvextrap} and density.
\end{proof}
\begin{remark}\label{rem:vecspdom}
If one can use an argument where extrapolation is only required once, then we may be able to replace the exponent in \eqref{eq:extratwice} by the smaller exponent
\[
\max\left(\frac{\frac{1}{r_1}}{\frac{1}{r_1}-\frac{1}{p_1}},\ldots,\frac{\frac{1}{r_m}}{\frac{1}{r_m}-\frac{1}{p_m}},\frac{1-\frac{1}{s}}{\frac{1}{p}-\frac{1}{s}}\right)
\]
which no longer depends on the exponents of the $\ell^{q_j}$ spaces. One way of doing this is by considering a vector-valued sparse domination rather than a scalar one. Such a sparse domination for the bilinear Hilbert transform is obtained in \cite{BM17}. See also \cite{HH14} where such ideas are used for vector-valued Calder\'on-Zygmund operators.
\end{remark}
\subsection{The bilinear Hilbert transform}
The bilinear Hilbert transform
\[
\BH(f_1,f_2)(x):=\pv\int_\R\!f_1(x-t)f_2(x+t)\,\frac{\mathrm{d}t}{t}
\]
is an integral operator falling outside of the theory of bilinear Calder\'on-Zygmund operators. It was introduced by A. Calder\'on and he wanted to know if it was bounded as an operator from $L^2\times L^\infty$ to $L^2$. This question was answered by Lacey and Thiele and they showed that $\BH$ is bounded $L^{p_1}\times L^{p_2}\to L^p$ for all $p_1,p_2\in(1,\infty]$ with $\frac{2}{3}<p<\infty$, $\frac{1}{p}=\frac{1}{p_1}+\frac{1}{p_2}$, see \cite{LT99}. It is an open problem whether one can remove the condition $\frac{1}{p}<\frac{3}{2}$ or not. However, in this range several weighted bounds and vector-valued extensions have been obtained, some of which we detail here.

Let $r_1,r_2,s\in(1,\infty)$. Then, under certain conditions on $r_1$, $r_2$, and $s$, the sparse domination
\[
|\langle\BH(f_1,f_2),g\rangle|\lesssim\sup_{\Sp}\Lambda_{(r_1,r_2,s'),\Sp}(f_1,f_2,g)
\]
was shown in \cite{CDO16}. These conditions can be formulated in the following equivalent ways:
\begin{lemma}\label{lem:equivbh}
Let $r_1,r_2,s\in(1,\infty)$. Then the following conditions are equivalent:
\begin{enumerate}[(i)]
\item \label{it:bhs1} We have $\max\left(\frac{1}{r_1},\frac{1}{2}\right)+\max\left(\frac{1}{r_2},\frac{1}{2}\right)+\max\left(\frac{1}{s'},\frac{1}{2}\right)<2;$
\item \label{it:bhs2} There exist $\theta_1,\theta_2,\theta_3\in[0,1)$ with $\theta_1+\theta_2+\theta_3=1$ so that
\[
\frac{1}{r_1}<\frac{1+\theta_1}{2},\qquad\frac{1}{r_2}<\frac{1+\theta_2}{2},\qquad\frac{1}{s}>\frac{1-\theta_3}{2}.
\]
\end{enumerate}
\end{lemma}
The sparse domination in terms of characterization \ref{it:bhs1} was obtained by Culiuc, Di Plinio and Ou in \cite{CDO16} and characterization \ref{it:bhs2} was used in \cite{BM17} where more general vector-valued sparse domination results were obtained.

Note that if we have $r_1,r_2,s\in(1,\infty)$ satisfying one of the equivalent conditions \ref{it:bhs1} or \ref{it:bhs2} and we have $p_1,p_2\in(1,\infty]$ with $\vec{p}>(\vec{r},s)$, then
\[
\frac{1}{p}=\frac{1}{p_1}+\frac{1}{p_2}<\max\left(\frac{1}{r_1},\frac{1}{2}\right)+\max\left(\frac{1}{r_2},\frac{1}{2}\right)<2-\max\left(\frac{1}{s'},\frac{1}{2}\right)\leq\frac{3}{2}
\]
so that we are still in the range of Lacey and Thiele.

From the sparse domination result for $\BH$, it was deduced in \cite{CDO16} that we have the weighted bounds $\BH:L^{p_1}(w_1^{p_1})\times L^{p_1}(w_1^{p_1})\to L^{p}(w^p)$ for all $p_1,p_2\in(1,\infty)$ with $\vec{p}>(\vec{r},s)$ in the Banach range $p>1$ and for all $\vec{w}\in A_{\vec{p},(\vec{r},s)}$. These weighted bounds were used in \cite{CM17} to obtain weighted and vector-valued estimates in the range $p\leq 1$ through extrapolation using products of $A_p$ classes. This result was extended in \cite{LMO18} where the full multilinear weight classes were used, but only the cases for finite $p_j$ were treated. However, their methods can be used to also obtain the cases with $p_j=\infty$ \cite{LMMOV19}. By applying Corollary~\ref{cor:sparseext} and Corollary~\ref{cor:vecext} we obtain the following result:
\begin{corollary}\label{cor:bhfinests}
Let $r_1,r_2,s\in(1,\infty)$ satisfy one of the equivalent conditions in Lemma~\ref{lem:equivbh}. Then for all $p_1,p_2\in(1,\infty]$ with $\vec{p}>(\vec{r},s)$ we have
\[
\|\BH\|_{L^{p_1}(w_1^{p_1})\times L^{p_2}(w_2^{p_2})\to L^p(w^p)}\lesssim[\vec{w}]_{\vec{p},(\vec{r},s)}^{ \max\left(\frac{\frac{1}{r_1}}{\frac{1}{r_1}-\frac{1}{p_1}},\frac{\frac{1}{r_2}}{\frac{1}{r_2}-\frac{1}{p_2}},\frac{1-\frac{1}{s}}{\frac{1}{p}-\frac{1}{s}}\right)}.
\]
for all $\vec{w}\in A_{\vec{p},(\vec{r},s)}$.

Moreover, for all $p_1,p_2,q_1,q_2\in(1,\infty]$ with $\vec{p},\vec{q}>(\vec{r},s)$ there is an increasing function $\phi_{\vec{p},\vec{q},\vec{r},s}$ such that
\begin{equation}\label{eq:bhwvecest}
\|\BH\|_{L^{p_1}(w_1^{p_1};\ell^{q_1})\times L^{p_2}(w_2^{p_2};\ell^{q_2})\to L^p(w^p;\ell^q)}\lesssim\phi_{\vec{p},\vec{q},\vec{r},s}([\vec{w}]_{\vec{p},(\vec{r},s)})
\end{equation}
for all $\vec{w}\in A_{\vec{p},(\vec{r},s)}$.
\end{corollary}
While Corollary~\ref{cor:vecext} gives us an expression for the increasing function $\phi_{\vec{p},\vec{q},\vec{r},s}$ in \eqref{eq:bhwvecest}, this estimate will not be sharp in general, see also Remark~\ref{rem:vecspdom}. Rather, a better quantitative estimate can be obtained if one applies our extrapolation result to weighted bounds that can be obtained from the vector-valued sparse domination result obtained in \cite[Theorem 1]{BM17}, but we do not pursue this further here.

Our result should be compared with \cite[Corollary 2.17]{LMO18} and \cite[Theorem 3]{BM18}. Qualitatively, we completely recover the results on weighted boundedness in \cite[Corollary 2.17]{LMO18} and extend it in the sense that we also include the cases where either $p_1$ or $p_2$ is equal to $\infty$ and where either $q_1$ or $q_2$ is equal to $\infty$, but this can also be done through their methods \cite{LMMOV19}. If, for example $p_1=\infty$, then we have $p_2=p$ and our scalar bound takes the form
\[
\|\BH(f_1,f_2)\|_{L^p(w^p)}\lesssim[\vec{w}]_{(\infty,p),(\vec{r},s)}^{ \max\left(\frac{\frac{1}{r_2}}{\frac{1}{r_2}-\frac{1}{p}},\frac{1-\frac{1}{s}}{\frac{1}{p}-\frac{1}{s}}\right)}\|f_1w_1\|_{L^\infty}\|f_2\|_{L^{p}(w_2^p)}
\]
for all $p\in(r_2,s)$ and all weights $w_1$, $w_2$ satisfying
\[
[\vec{w}]_{(\infty,p),(\vec{r},s)}=\sup_{Q}\langle w_1^{-1}\rangle_{r_1,Q}\langle w_2^{-1}\rangle_{\frac{1}{\frac{1}{r_2}-\frac{1}{p}}}\langle w_1w_2\rangle_{\frac{1}{\frac{1}{p}-\frac{1}{s}}}<\infty.
\]
This is also slightly more general than the weighted bounds in \cite[Corollary 3]{BM17} in this endpoint case since they only formulate their result in the case $p_1=\infty$ when $w_1=1$ (or more generally, $p_j=\infty$ when $w_j=1$), but their methods do allow for this more general case.

The result \cite[Theorem 3]{BM18} asserts that if $p_1,p_2,q_1,q_2\in(1,\infty]$ satisfy $\vec{p},\vec{q}>(\vec{r},s)$ for $r_1,r_2,s\in(1,\infty)$ satisfying one of the equivalent properties of Lemma~\ref{lem:equivbh}, then we have
\begin{equation}\label{eq:bm18vect}
\|\BH\|_{L^{p_1}(\ell^{q_1})\times L^{p_2}(\ell^{q_2})\to L^p(\ell^q)}<\infty.
\end{equation}
This result is completely recovered in Corollary~\ref{cor:bhfinests} in the unweighted version of \eqref{eq:bhwvecest}.

By again extrapolating from the weighted vector-valued bounds we can also consider iterated $\ell^q$ spaces in our results. For example, by applying Theorem~\ref{thm:vvextrap} to the weighted vector valued bounds \eqref{eq:bhwvecest}, one can obtain
\[
\BH:L^{p_1}(\ell^2(\ell^\infty)) \times L^{p_2}(\ell^\infty(\ell^2))\to L^p(\ell^2(\ell^2))
\]
for all $p_1,p_2\in(1,\infty]$ with $\frac{2}{3}<p<\infty$. Such bounds were already obtained in \cite{BM16} through the helicoidal method, but could not be obtained through earlier extrapolation results. More precisely, to obtain this result through extrapolation one needs to be able to extrapolate away from weighted $L^\infty$ estimates which is one of our novelties. These type of multiple vector-valued bounds can be applied to prove boundedness results of operators such as the tensor product of $\BH$ and paraproducts and we refer the reader to \cite{BM16} for an overview of such operators.
\subsection{Endpoint extrapolation results}
Finally, we shall discuss some of the endpoint estimates one can extrapolate from.

The following is an extrapolation result involving weak-type estimates. The trick used to obtain this result is well-known and can be found already in \cite{GM04}.
\begin{theorem}[Weak type extrapolation]
Let $(f_1,\ldots,f_m,h)$ be an $m+1$-tuple of measurable functions and let $r_1,\ldots,r_m\in(0,\infty)$, $s\in(0,\infty]$. Suppose that for some $q_1,\ldots,q_m\in(0,\infty]$ with $\vec{q}\geq(\vec{r},s)$ there is an increasing function $\phi_{\vec{q}}$ such that
\begin{equation}\label{eq:weakmultextrapinit}
\|h\|_{L^{q,\infty}(w^q)}\leq\phi_{\vec{q}}([\vec{w}]_{\vec{q},(\vec{r},s)})\prod_{j=1}^m\|f_j\|_{L^{q_j}(w_j^{q_j})}
\end{equation}
for all $\vec{w}\in A_{\vec{q},(\vec{r},s)}$.

Then for all $p_1\,\ldots,p_m\in(0,\infty]$ with $\vec{p}>(\vec{r},s)$ there is an increasing function $\phi_{\vec{p},\vec{q},\vec{r},s}$ such that
\begin{equation}\label{eq:weakmultextrapend}
\|h\|_{L^{p,\infty}(w^p)}\leq\phi_{\vec{p},\vec{q},\vec{r},s}([\vec{w}]_{\vec{p},(\vec{r},s)})\prod_{j=1}^m\|f_j\|_{L^{p_j}(w_j^{p_j})}
\end{equation}
for all $\vec{w}\in A_{\vec{p},(\vec{r},s)}$. More explicitly, we can take
\begin{equation}\label{eq:weakmultextraquant}
\phi_{\vec{p},\vec{q},\vec{r},s}(t)=2^{\frac{m^2}{r}}\phi_{\vec{q}}\Big(C_{\vec{p},\vec{q},\vec{r},s}t^{ r\max\left(\frac{\frac{1}{r_1}-\frac{1}{q_1}}{\frac{1}{r_1}-\frac{1}{p_1}},\ldots,\frac{\frac{1}{r_m}-\frac{1}{q_m}}{\frac{1}{r_m}-\frac{1}{p_m}},\frac{\frac{1}{q}-\frac{1}{s}}{\frac{1}{p}-\frac{1}{s}}\right)}\Big)^{\frac{1}{r}},
\end{equation}
where $\frac{1}{r}=\sum_{j=1}^m\frac{1}{r_j}$.
\end{theorem}
\begin{proof}
Let $\lambda>0$ and set $E_\lambda:=\{x\in\R^n:|h(x)|>\lambda\}$. We define
\[
h_\lambda:=\lambda \chi_{E_\lambda}
\]
and note that by \eqref{eq:weakmultextrapinit} we have
\[
\|h_\lambda\|_{L^q(w^q)}=\lambda \big(w^q(E_\lambda)\big)^{\frac{1}{q}}\leq\|h\|_{L^{q,\infty}(w^q)}\leq\phi_{\vec{q}}([\vec{w}]_{\vec{q},(\vec{r},s)})\prod_{j=1}^m\|f_j\|_{L^{q_j}(w_j^{q_j})}
\]
Thus, by applying Theorem~\ref{thm:qmlre} to the $m+1$-tuple $(f_1,\ldots,f_m,h_\lambda)$ we conclude that for all $p_1\,\ldots,p_m\in(0,\infty]$ with $\vec{p}>(\vec{r},s)$ there is an increasing function $\phi_{\vec{p},\vec{q},\vec{r},s}$ such that
\[
\|h_\lambda\|_{L^p(w^p)}\leq\phi_{\vec{p},\vec{q},\vec{r},s}([\vec{w}]_{\vec{p},(\vec{r},s)})\prod_{j=1}^m\|f_j\|_{L^{p_j}(w_j^{p_j})}
\]
for all $\vec{w}\in A_{\vec{p},(\vec{r},s)}$, with $\phi_{\vec{p},\vec{q},\vec{r},s}$ given by \eqref{eq:weakmultextraquant}. As $\lambda>0$ was arbitrary, noting that $\sup_{\lambda>0}\|h_\lambda\|_{L^p(w^p)}=\|h\|_{L^{p,\infty}(w^p)}$ proves \eqref{eq:weakmultextrapend}. The assertion follows.
\end{proof}
As a consequence we can extrapolate from weak lower endpoint estimates in cases where strong bounds are not available. Passing to the full-range case where $r_1=\cdots=r_m=1$ and $s=\infty$, writing $\vec{1}$ for the vector consisting of $m$ components all equal to $1$, we obtain the following corollary:
\begin{corollary}
Let $(f_1,\ldots,f_m,h)$ be an $m+1$-tuple of measurable functions and suppose that there is an increasing function $\phi$ such that
\[
\|h\|_{L^{\frac{1}{m},\infty}(w^{\frac{1}{m}})}\leq\phi([\vec{w}]_{\vec{1},(\vec{1},\infty)})\prod_{j=1}^m\|f_j\|_{L^1(w_j)}
\]
for all $\vec{w}\in A_{\vec{1},(\vec{1},\infty)}$.

Then for all $p_1\,\ldots,p_m\in(1,\infty]$ with $\frac{1}{p}>0$ there is an increasing function $\phi_{\vec{p}}$ such that
\[
\|h\|_{L^{p,\infty}(w^p)}\leq\phi_{\vec{p}}([\vec{w}]_{\vec{p},(\vec{1},\infty)})\prod_{j=1}^m\|f_j\|_{L^{p_j}(w_j^{p_j})}
\]
for all $\vec{w}\in A_{\vec{p},(\vec{1},\infty)}$. More explicitly, we can take
\[
\phi_{\vec{p}}(t)=2^{m^3}\phi\Big(C_{\vec{p}}t^p\Big)^m.
\]
\end{corollary}

On the other hand, we can also extrapolate from the upper endpoints. An application of Theorem~\ref{thm:qmlre} in the $s=\infty$ case with $q_1=\cdots=q_m=\infty$, writing $\vec{\infty}$ for the vector consisting of $m$ components all equal to $\infty$, yields the following:
\begin{theorem}[Upper endpoint extrapolation]\label{thm:upendextrap}
Let $(f_1,\ldots,f_m,h)$ be an $m+1$-tuple of measurable functions and let $r_1,\ldots,r_m\in(0,\infty)$. Suppose that there is an increasing function $\phi$ such that
\[
\|hw\|_{L^{\infty}}\leq\phi([\vec{w}]_{\vec{\infty},(\vec{r},\infty)})\prod_{j=1}^m\|f_jw_j\|_{L^\infty}
\]
for all $\vec{w}\in A_{\vec{\infty},(\vec{r},\infty)}$.

Then for all $p_1\,\ldots,p_m\in(0,\infty]$ with $\vec{p}>\vec{r}$, there is an increasing function $\phi_{\vec{p},\vec{r}}$ such that
\[
\|h\|_{L^p(w^p)}\leq\phi_{\vec{p},\vec{r}}([\vec{w}]_{\vec{p},(\vec{r},\infty)})\prod_{j=1}^m\|f_j\|_{L^{p_j}(w_j^{p_j})}
\]
for all $\vec{w}\in A_{\vec{p},(\vec{r},\infty)}$. More explicitly, we can take
\[
\phi_{\vec{p},\vec{r}}(t)=2^{\frac{m}{r}}\phi\Big(C_{\vec{p},\vec{r}}t^{r\max_{j=1,\ldots,m}\left\{\frac{\frac{1}{r_j}}{\frac{1}{r_j}-\frac{1}{p_j}}\right\}}\Big)^{\frac{1}{r}},
\]
where $\frac{1}{r}=\sum_{j=1}^m\frac{1}{r_j}$.
\end{theorem}
An interesting application is related to the space $\BMO$ of functions of bounded mean oscillation. We define the sharp maximal operator $M^\#$ by
\[
M^\#f=\sup_{Q}\langle|f-\langle f\rangle_{1,Q}|\rangle_{1,Q}\chi_Q
\]
for locally integrable functions $f$, where the supremum is taken over all cubes $Q\subseteq\R^n$. The classical definition of $\BMO$ can be given in terms of $M^\#$ by saying a measurable function $f$ is in $\BMO$ if $M^\#f\in L^\infty$, with $\|f\|_{\BMO}:=\|M^\#f\|_{L^\infty}$. The way we have dealt with weighted estimates in $L^\infty$ so far suggests the following definition of a weighted version of the $\BMO$ space:
\begin{definition}
Given a weight $w$, we define the space $\BMO(w)$ as those locally integrable functions $f$ such that
\[
\|f\|_{\BMO(w)}:=\|(M^\# f)w\|_{L^\infty}<\infty.
\]
\end{definition}
Weighted $\BMO$ spaces also appeared in the work of Muckenhoupt and Wheeden in \cite{MW76}, and they showed that the estimate
\begin{equation}\label{eq:bmoextrap}
\|Tf\|_{\BMO(w)}\lesssim\|fw\|_{L^\infty},
\end{equation}
with an explicit constant depending on $w$, is satisfied when $T$ is the Hilbert transform, if and only if $w^{-1}\in A_1$. We recall here that the condition $w^{-1}\in A_1$ is equivalent to our condition $w\in A_{\infty,(1,\infty)}$ with $[w]_{\infty,(1,\infty)}=[w^{-1}]_{A_1}$. Later it was shown by Harboure, Mac\'ias and Segovia in \cite{HMS88} that one can extrapolate from the estimate \eqref{eq:bmoextrap} for an operator $T$ to obtain that $T$ is bounded on $L^p(w^p)$ for all $w^p\in A_p$. As a consequence of Theorem~\ref{thm:upendextrap} we obtain a qualitative multilinear version of this result.
\begin{corollary}[Extrapolation from $\BMO$ estimates]
Let $T$ be an $m$-(sub)linear operator and let $r_1,\ldots,r_m\in(0,\infty)$. Suppose that there is an increasing function $\phi$ such that
\[
\|T(f_1,\ldots,f_m)\|_{\BMO(w)}\leq\phi([\vec{w}]_{\vec{\infty},(\vec{r},\infty)})\prod_{j=1}^m\|f_jw_j\|_{L^\infty}
\]
for all $\vec{w}\in A_{\vec{\infty},(\vec{r},\infty)}$ and all $f_j$ with $f_jw_j\in L^\infty$.

Then for all $p_1\,\ldots,p_m\in(0,\infty]$ with $\vec{p}>\vec{r}$, there is an increasing function $\phi_{\vec{p},\vec{r}}$ such that
\[
\|T(f_1,\ldots,f_m)\|_{L^p(w^p)}\leq\phi_{\vec{p},\vec{r}}([\vec{w}]_{\vec{p},(\vec{r},\infty)})\prod_{j=1}^m\|f_j\|_{L^{p_j}(w_j^{p_j})}
\]
for all $\vec{w}\in A_{\vec{p},(\vec{r},\infty)}$ and all $f_j\in L^{p_j}(w_j^{p_j})$, whenever the left-hand side is finite.
\end{corollary}
\begin{proof}
We apply Theorem~\ref{thm:upendextrap} to the $m+1$-tuples $(f_1,\ldots,f_m,M^\#(T(f_1,\ldots,f_m)))$. Then we find that for all $p_1\,\ldots,p_m\in(0,\infty]$ with $\vec{p}>\vec{r}$, there is an increasing function $\phi_{\vec{p},\vec{r}}$ such that
\begin{equation}\label{eq:sharpext}
\|M^\#(T(f_1,\ldots,f_m))\|_{L^p(w^p)}\leq\phi_{\vec{p},\vec{r}}([\vec{w}]_{\vec{p},(\vec{r},\infty)})\prod_{j=1}^m\|f_j\|_{L^{p_j}(w_j^{p_j})}
\end{equation}
for all $\vec{w}\in A_{\vec{p},(\vec{r},\infty)}$ and all $f_j\in L^{p_j}(w_j^{p_j})$. By the classical Fefferman-Stein inequality for the sharp maximal operator, see \cite{FS72}, we find that
\[
\|T(f_1,\ldots,f_m)\|_{L^p(w^p)}\lesssim\|M^\#(T(f_1,\ldots,f_m))\|_{L^p(w^p)},
\]
for $p>1$, with implicit constant depending on the $A_\infty$ constant of $w^p$, which is bounded by an increasing function of $[w]_{p,(r,\infty)}$, where $\frac{1}{r}=\sum_{j=1}^m\frac{1}{r_j}$, see also \cite[Chapter 7]{Gr14c}. Since $[w]_{p,(r,\infty)}\leq[\vec{w}]_{\vec{p},(\vec{r},s)}$ by H\"older's inequality, the result for $p>1$ follows from \eqref{eq:sharpext}. By extrapolating again, we also obtain the cases $p\leq 1$, proving the assertion.
\end{proof}
Examples of multilinear operators satisfying weak-type and $\BMO$ endpoint estimates are multilinear Calder\'on-Zygmund operators, see also \cite[Section 7.4.1]{Gr14m}. Weighted estimates in these situations can be found in \cite{LOPTT09}.

\bibliographystyle{alpha-sort}
\newcommand{\etalchar}[1]{$^{#1}$}

\end{document}